\documentclass[11pt]{article}
\usepackage{hyperref} 
\usepackage[a4paper, margin = 2.2cm]{geometry}
\usepackage[english]{babel}
\usepackage{latexsym,amsmath,enumerate,graphics,enumerate,amsthm,tikz,hyperref,float,comment}
\usepackage[mathscr]{euscript}
\usepackage[affil-it]{authblk}
\usepackage{enumerate,amsthm,dsfont,pstricks}
\usepackage{latexsym,amsmath,amssymb,amscd,wrapfig,graphicx}
\usepackage{wrapfig,graphicx,curves}
\usepackage{tikz}
\usetikzlibrary{hobby}
\usepackage{booktabs}
\usepackage{multirow}
\usepackage{siunitx}

\newtheorem{theorem}{Theorem}[section]
\newtheorem{lemma}[theorem]{Lemma}
\newtheorem{proposition}[theorem]{Proposition}
\newtheorem{corollary}[theorem]{Corollary}

\theoremstyle{definition}
\newtheorem{definition}[theorem]{Definition}

\newtheorem{remark}[theorem]{Remark}
\theoremstyle{remark}

\let\phi=\varphi
\def\epsilon{\varepsilon}

\def\0{\mathbf{0}}

\def\dimh{\dim_{\mathcal{H}}}

\numberwithin{equation}{section}

\def\dimh{\dim_{\mathcal{H}}}

\let\epsilon=\varepsilon

\makeatletter
\def\@maketitle{%
  \newpage
  \null
  \vskip 0em%
  \begin{center}%
  \let \footnote \thanks
    {\Large\bfseries \@title \par}%
    \vskip 0em%
    {\normalsize
      \lineskip 0em%
      \begin{tabular}[t]{c}%
        \@author
      \end{tabular}\par}%
    \vskip 0em%
    {\normalsize \@date}%
  \end{center}%
  \par
  \vskip 0em}
\makeatother

\begin{document}
\date{}
\title{\LARGE \sc On the structure of the dimension spectrum for continued fraction expansions}

\author[1]{Painos Chitanga%
\thanks{Email: \texttt{pc441@kent.ac.uk}, Painos Chitanga gratefully acknowledges the support of the EPSRC grant  EP/V520093/1}}
\affil[1]{School of Mathematics, Statistics \& Actuarial Science,
University of Kent,\newline Canterbury CT2 7NX, UK}

\author[1]{Bas Lemmens%
\thanks{Email: \texttt{B.Lemmens@kent.ac.uk}}}

\author[2]{Roger Nussbaum%
\thanks{Email: \texttt{nussbaum@math.rutgers.edu}}}
\affil[2]{Department of Mathematics, Rutgers University, Piscataway, NJ, 08854, USA}

\maketitle
\date{}

\begin{abstract}
We analyse the (Hausdorff) dimension spectrum of continued fractions expansions with coefficients restricted to infinite subsets of $ \mathbb{N}$. We prove that the set of powers $P_q=\{q^n\colon n\in \mathbb{N}\}$ has full dimension spectrum for each integer $q\geq 2$, answering a question by Chousionis, Leykekhman and Urba\'nski. On the other hand, we show that the dimension spectrum for   $P^*_q=\{q^n\colon n\in \mathbb{N}\}\cup\{1\}$ has many gaps and regions where it is nowhere dense. We also investigate the case where $A$ is generated by a monomial, $M_q=\{n^q\colon n\in\mathbb{N}\}$. For $M_q$ we prove that  the dimension spectrum is full for $q\in\{1,2,3,4,5\}$, and it has a gap for each $q\geq 6$. Furthermore we show  for $q\in\{6,7,8\}$ that the dimension spectrum of $M_q$ is the disjoint union of two nontrivial  closed intervals, and it is the disjoint union of three nontrivial closed intervals for $q \in\{9,10\}$. For $q\geq 11$ we show that the dimension spectrum of $M_q$ consists of finitely many disjoint  nontrivial closed  intervals. The results concerning $M_q$ extend existing results for $q=1$ and $q=2$. In our analysis we employ Perron-Frobenius (transfer) operators, and numerical tools developed by Falk and Nussbaum that give rigorous  estimates for the Hausdorff dimension for continued fractions expansions.
\end{abstract}




\section{Introduction}
In this paper we investigate for  infinite sets $A\subseteq \mathbb{N}$ the set of continued fraction expansions,  
\[
J_A=\{ x\in (0,1)\colon x =[a_1, a_2, a_3, \cdots] \mbox{ with $a_i \in A$ for all $i$}\},
\]
where 
\[
	[a_1, a_2, a_3,\cdots] = \cfrac{1}{a_1 + \cfrac{1}{a_2+  \frac{1}{a_3 + \cdots}}}\,.
	\] 
	  These sets have a fractal nature and their Hausdorff dimension, denoted $\dim_{\mathcal{H}}(J_A)$, has  been studied extensively, see for instance \cite{Bu1,Bu2,FN1,FN2,Go,He1,He2,JP,MU1,MU2,N1,NPVL}.  For the basic concepts regarding continued fractions and Hausdorff dimension we refer the reader to \cite{Fa,Hen}. 

We analyse the {\em (Hausdorff) dimension spectrum of  $A\subseteq \mathbb{N}$}, 
\[
\mathrm{DS}(A) = \{\dimh(J_B)\colon B\subseteq A\},
\] 
which was studied recently by Chousionis, Leykekhman and Urba\'nski in \cite{CLU1,CLU2} for different infinite subsets $A$ of $\mathbb{N}$, see also \cite{CLUW,DS,Ju}. The case where $A=\mathbb{N}$ was considered  earlier by Kesseb\"ohmer and Zhu \cite{KZ}, who showed that it  has  full dimension spectrum, i.e., $  \mathrm{DS}(\mathbb{N})=[0,1]$, which confirmed a conjecture by Hensley \cite{He3} and Mauldin and Urba\'nski \cite{MU2} known as the  Texan Conjecture, see also \cite{Je}.  The main motivation for the results presented here comes from  \cite{CLU2}  in which the dimension spectrum of the set of powers of  integers  $q\geq 2$, the set of squares, and the set of primes was analysed among other sets. 

In this paper we investigate the dimension spectrum for a variety of natural choices of $A$ including the set of powers of integers $q\geq 2$:  $P_q=\{q^n\colon n\in\mathbb{N}\}$ and $P_q^*= P_q\cup \{1\}$. In \cite[Theorem 1.4]{CLU2} the dimension spectrum of $P_q$ was considered, and for  $q\geq 2$ it was shown there that there exists an $s(q)>0$ such that 
\[
 [0,\min\{s(q), \dim_{\mathcal{H}}(J_{P_q})\}]\subseteq \mathrm{DS}(P_q).
 \]
 We prove that $P_q$ has full dimension spectrum for all $q\geq 2$, answering a question from \cite{CLU2}. In fact, we will establish the following more general result. 
 \begin{theorem}\label{thm1} If $A=\{a_1,a_2,\ldots\}\subset \mathbb{N}$ with $2\leq a_1<a_2<\ldots$ and $a_na_m\geq a_{n+m}$ for all $m\geq n\geq 1$ with  $m\geq 2$, then  
\begin{equation}\label{full} 
 [0,\dim_{\mathcal{H}}(J_{A})]= \mathrm{DS}(A).
 \end{equation} 
 \end{theorem}
This result also implies several results from \cite{CLU2},  in particular \cite[Theorem 4.11]{CLU2} and  \cite[Theorem 1.2]{CLU2}. We summarise some of its consequences in the following corollary. 
 \begin{corollary}\label{cor1} The following subsets of $\mathbb{N}$ have full dimension spectrum:
 \begin{enumerate}[(a)]
 \item $A=\{a+bn\colon n \in\mathbb{N}\}$  for integers $0\leq a<b$ with $b\geq 2$. 
 \item $A=\{a+bn\colon n \in\mathbb{N}\cup\{0\}\}$  for integers $2\leq a<b$. 
 \item $A_{\mathrm{primes}}=\{p\in\mathbb{N}\colon p \mbox{ prime}\}$. 
 \item $P_q=\{q^n\colon n\in\mathbb{N}\}$ for $q\geq 2$. 
 \end{enumerate}
 \end{corollary} 
 The condition in Theorem \ref{thm1}  for the set of primes  is an old result due to Ishikawa \cite[Satz III]{Is}.  For the other statements in the corollary the condition can be easily verified.

 As we shall see, the fullness of the dimension spectrum of $P_q$ is in stark contrast with the dimension spectrum of $P_q^*$, which has many gaps.  
 More specifically, given $q\geq 2$ and $k\geq 0$ let 
 \[
 I_k=\{1,\ldots,q^k\}\mbox{\quad and \quad } T_k =\{q^{k+1},q^{k+2},\ldots\}, 
 \]
 and  set
 \[
 \mu^k = \dim_{\mathcal{H}}(J_{I_{k-1}\cup T_{k}})=\dimh(J_{P^*_q\setminus\{q^k\}})\mbox{\quad and \quad} \nu^k = \dim_{\mathcal{H}}(J_{I_{k}})\mbox{\qquad for } k\geq 0.
 \]
 Note that $\mu^0= \dimh(J_{P_q})$ and $\nu^0 =\dimh(J_{\{1\}}) =0$. 
  We have the following result.
 \begin{theorem}\label{thm2}
 For all $q\geq 3$ and $k \geq 1$,  
\begin{enumerate}[(i)]
\item $\mu^k<\nu^k$ and $(\mu^k,\nu^k)\cap \mathrm{DS}(P_q^*)=\emptyset$. 
\item $\mathrm{DS}(P_q^*)$ is nowhere dense in $(\nu^k,\mu^{k+1})$. 
\end{enumerate} 
For $q=2$,  assertions (i) and (ii) hold for all $k\geq 2$. 
\end{theorem}
Furthermore, the dimension spectrum of $P^*_q$  contains an initial nontrivial interval.   
\begin{theorem}\label{thm3} The interval $[0,\frac{\ln 2}{2\ln q}]$ is contained in $\mathrm{DS}(P_q^*)$ for each $q\geq 2$. 
\end{theorem}  
We note that in \cite[Theorem 1.1]{CLU1} it was shown that  the dimension spectrum of $A$ always contains an interval $[0,\sigma_0]$, where $\sigma_0= \inf\{s>0\colon   \sum_{n\in A} n^{-2s}<\infty\}$ is the {\em finiteness parameter}.
In the case of $P_q$ and $P_q^*$, however, $\sigma_0 =0$. 

We will also analyse the dimension spectrum for sets generated by a monomial, $M_q=\{n^q\colon n\in\mathbb{N}\}$, and prove the following result.  
\begin{theorem}\label{Mq} The dimension spectrum of $M_q$ satisfies: 
\begin{enumerate}[(i)]
\item For $q\in\{1, 2, 3, 4, 5\}$ we have that $\mathrm{DS}(M_q) = [0,\dimh(J_{M_q})]$. 
\item For  $q\geq 6$ we have 
\[\dimh(J_{M_q\setminus\{2^q\}})<\dimh(J_{\{1,2^q\}})
\]
and  $\mathrm{DS}(M_q)\cap (\dimh(J_{M_q\setminus\{2^q\}}), \dimh(J_{\{1,2^q\}}))$ is empty.
\item For $q\in \{6,7,8\}$ we have that 
\[
 \mathrm{DS}(M_q) = [0,\dimh(J_{M_q\setminus\{2^q\}})]\cup [\dimh(J_{\{1,2^q\}}),\dimh(J_{M_q})].
\] 
\item For $q\in\{9,10\}$ we have that $\dimh(J_{M_q\setminus\{3^q\}})<\dimh(J_{\{1, 2^q, 3^q\}})$ and 
 \[
 \mathrm{DS}(M_q) = [0,\dimh(J_{M_q\setminus\{2^q\}})]\cup [\dimh(J_{\{1,2^q\}}),\dimh(J_{M_q\setminus\{3^q\}})] \cup [\dimh(J_{\{1, 2^q, 3^q\}}),\dimh(J_{M_q})].
 \]
 \item  $\mathrm{DS}(M_q)$ is the disjoint union of finitely many nontrivial closed intervals for each $q\in\mathbb{N}$. 
\end{enumerate}
\end{theorem}
The case  $q=1$ is the Texan Conjecture established in \cite{KZ}, and the case $q=2$ was treated in \cite[Theorem 1.3]{CLU2}. To prove the final statement in Theorem \ref{Mq} we will establish a general criterion on subsets $A$ of $\mathbb{N}$ that implies that its dimension spectrum  consists of finitely many nontrivial disjoint closed interval, see Theorem \ref{intervals}.  In \cite{CLU1} in the remark after Conjecture 1.3 the authors raised the question if there exists a conformal  iterated function systems with a prescribed finite number of connected components. Theorem \ref{Mq} seems to suggest that this might be realised using the sets $M_q$, but we have no proof at present.

Throughout the paper the $a_n$'s and $q$ will be integers, although this is not strictly required for several of the statements presented, in particular in Theorem \ref{thm1}. In these instances it sufficient to know that the maps $\theta_n\colon x\mapsto (a_n+x)^{-1}$ have disjoint ranges on the invariant set. 

In our analysis we will use  Perron-Frobenius (or transfer) operators. Given $F\subset \mathbb{N}$ finite and $s\geq 0$, the {\em Perron-Frobenius operator}, $L_{s,F}\colon C([0,1])\to C([0,1])$, on the Banach space of real continuous functions on $[0,1]$ is given by 
\[
(L_{s,F}f)(x) = \sum_{n\in F} \left(\frac{1}{n+x}\right)^{2s} f\left(\frac{1}{n+x}\right)\mbox{\quad for $x\in [0,1]$},
\]
which is a positive bounded linear operator on $C([0,1])$. Here positive means that if $f\in C([0,1])$ with $f(x)\geq 0$ for all $x\in [0,1]$, then $(L_{s,F}f)(x)\geq 0$ for all $x\in [0,1]$.  

The operator $L_{s,F}$ can be considered on other Banach spaces. For instance on the real Banach space $C^k([0,1])$ consisting of $k$-times continuously differentiable functions $f\colon [0,1]\to \mathbb{R}$ (respectively the complex Banach space $C_\mathbb{C}^k([0,1])$ with functions $f\colon [0,1]\to \mathbb{C}$). It can also be considered  on the Banach space $C^{0,\alpha} ([0,1])$ (respectively $C^{0,\alpha}_\mathbb{C}([0,1])$) consisting of real (complex) functions on $[0,1]$ that  are H\"older continuous with H\"older exponent $0<\alpha\leq 1$.  Indeed, $L_{s,F}$  is a bounded real linear operator  from $C^{0,\alpha}([0,1])$ to itself, and also from  $C^k([0,1])$ to itself. 
The operator can be extended in the usual way to a complex linear operator to $C_\mathbb{C}^{0,\alpha}([0,1])$ and also to $C^k_\mathbb{C}([0,1])$. 
If $L_{s,F}$ is considered as a  bounded complex linear operator on  $C_\mathbb{C}^{0,\alpha}([0,1])$ or $C^k_\mathbb{C}([0,1])$, we shall abuse notation and write $\sigma(L_{s,F})\subseteq \mathbb{C}$ to denote the spectrum of $L_{s,F}$. 


Rather than using the pressure function and Bowen's parameter from thermodynamic formalism \cite{MU}, which involve the iterates of the operator $L_{s,F}$ to analyse its spectral radius, we approximate the strictly positive eigenvector of $L_{s,F}$ to estimate the spectral radius. In particular, we use the eigenvector given in Lemma \ref{eigenvector}, which is surprisingly powerful. In this context it is convenient to use the next theorem, which  is a special case of more general theorems concerning the spectral theory of Perron-Frobenius type operators that can be found in: \cite[Theorem 3.1]{FN2},  \cite[Section 2.2]{LN}, \cite[Theorem 5.4]{N0}, and \cite[Theorem 6.5]{N1}. It  was also used in \cite{CLU2} where one can find a detailed discussion about its connection to known results in thermodynamic formalism, see \cite[Remark 3.1]{CLU2}.  The theory of positive operators plays an essential role in many areas of ergodic theory, see for instance \cite{Ba} for related results. 

\begin{theorem}\label{thm:FN}
For $F\subset \mathbb{N}$ finite, with  $\gamma=\min\{n\colon n\in F\}$,  $s>0$, $0<\alpha\leq 1$ and $k\in\mathbb{N}$ the following assertions hold.
\begin{enumerate}[(i)]
\item If  $L_{s,F}$ is considered as an operator from $C^{0,\alpha}([0,1])$ to itself (respectively from $C^k([0,1])$ to itself), then it has a strictly positive eigenvector $v_{s,F}\in C^{0,\alpha}([0,1])$  (respectively $v_{s,F}\in C^k([0,1])$) with corresponding eigenvalue $\lambda_{s,F}>0$. The eigenvector $v_{s,F}$ is unique up to  scaling, and $\lambda_{s,F}$ is independent of $\alpha$ and $k$ and equals the spectral radius of $L_{s,F}\colon C^{0,\alpha}([0,1])\to C^{0,\alpha}([0,1])$ (respectively $L_{s,F}\colon C^k([0,1])\to C^k([0,1])$). In particular, $v_{s,F}\in C^k([0,1])$ for all $k\in\mathbb{N}$, hence it is a $C^\infty$-function. It is also the unique positive eigenvector of $L_{s,F}\colon C([0,1])\to C([0,1])$ and $\lambda_{s,F}$ is the spectral radius, denoted $r(L_{s,F})$, of $L_{s,F}\colon C([0,1])\to C([0,1])$. 
\item The spectrum $\sigma(L_{s,F})\subseteq \mathbb{C}$ of $L_{s,F}\colon C^{0,\alpha}([0,1])\to C^{0,\alpha}([0,1])$ (or $L_{s,F}\colon C^k([0,1])\to C^k([0,1])$) satisfies
\[
\sup\left\{\frac{|z|}{\lambda_{s,F}}\colon z\in \sigma(L_{s,F})\setminus\{\lambda_{s,F}\}\right\}<1.
\]
\item The function $s\mapsto \lambda_{s,F}$ is strictly decreasing and continuous.
\item The function $v_{s,F}$ is a decreasing on $[0,1]$ and 
\[
-\frac{2s}{\gamma}\leq \frac{v'_{s,F}(x)}{v_{s,F}(x)}<0\mbox{\qquad for all }x\in [0,1].
\]
\item  The unique value $s$ such that $\lambda_{s,F} =1$ is equal to $\dimh(J_F)$.
\end{enumerate} 
\end{theorem}
As noted in \cite{FN2} the inequality in the fourth assertion in Theorem \ref{thm:FN} implies that 
\begin{equation}\label{ineqv_s}
v_{s,F}(x)\leq v_{s,F}(y)e^{\frac{2s|x-y|}{\gamma}}\mbox{\qquad for all }x,y\in[0,1].
\end{equation}

\begin{remark}
The fact, mentioned in the first assertion of Theorem \ref{thm:FN}, that the strictly positive (normalised) eigenvector   $v_{s,F}$ of $L_{s,F}\colon C([0,1])\to C([0,1])$  is unique,  is not proved in the literature to the best of our knowledge, but holds for a much larger class of Perron-Frobenius type operators than the operators $L_{s,F}$. As we will not require this fact here, we omit the proof.  
\end{remark}

\section{Preliminaries}

In  this section we recall some preliminary results that we will use throughout the paper.   For $a<b$, the Banach space $(C([a,b]),\|\cdot\|_\infty)$ is a complete order-unit space with cone $C([a,b])_+=\{f\in C([a,b])\colon f(x)\geq 0\mbox{ for all }x\in[a,b]\}$ and order-unit $u\colon x\mapsto 1$ for all $x$. So the partial ordering on $C([a,b])$ is given by $f\leq g$ if $f(x)\leq g(x)$ for all $x\in [a,b]$. 
	\begin{lemma}\label{scaling}
		Let $f,g \in C([a,b])$ be strictly positive. For each $0< \lambda < 1$,  there exists a  $\mu \in (\lambda,1)$ such that $ f + \lambda g \leq \mu(f + g)$.
	Likewise,  for each $\lambda > 1$,  there exists a $\mu \in (1, \lambda]$ such that 
		$ \mu(f + g) \leq  f + \lambda g$.
	\end{lemma}
	\begin{proof}
		Since $f$ and $g$ are strictly positive on $[a,b]$, the function $h(x) = \frac{f(x) + \lambda g(x)}{ f(x) + g(x)}$ 
		is well defined, strictly positive, and continuous. So, $h$ attains a maximum, say at  $x_0\in [a,b]$. Set $ \mu  = h(x_0)>0$. Then 
		\[ \mu = h(x_0) = \frac{f(x_0) + \lambda g(x_0)}{f(x_0) + g(x_0)} < 1.\]
		Thus, $\mu < 1$ and $f + \lambda g \leq \mu(f + g)$.
		As $\lambda (f(x) + g(x)) <  f(x) + \lambda g(x) \leq \mu (f(x) + g(x))$ for all $x\in [a,b]$, we also have  that $\lambda < \mu$.
		
		The second assertion can be derived in the same way by considering the minimum of $h$. 
	\end{proof}
	Recall that the spectral radius, $r(L)$, of a bounded linear operator $L\colon C([a,b])\to C([a,b])$ satisfies $r(L) = \lim_k \|L^k\|^{1/k}$, see \cite[p.197]{Con}. 
 The following basic fact is useful to estimate the spectral radius of the positive operators $L_{s,F}$ and will be used throughout.
	\begin{lemma}\label{basic} 
	Suppose that $L \colon C([a,b]) \to C([a,b])$ is a positive linear operator. If $w\in C([a,b])$ is  strictly positive and  $\alpha w\leq L w\leq \beta w$, then $\alpha\leq r(L)\leq \beta$. 
	\end{lemma}
	\begin{proof} Let $u\colon x\mapsto 1$ be the order-unit. As $L$ is positive, we have that $\|L^k\| = \|L^k u\|_\infty$. Moreover, there exists a $\mu,\nu>0$ such that $\mu w\leq u\leq \nu w$. Thus, 
	$\mu\alpha^k w\leq \mu L^k w\leq L^k u \leq \nu L^k w\leq \nu\beta^k w$, so that $\mu \alpha^k\|w\|_\infty\leq \|L^k\|\leq \nu\beta^k\|w\|_\infty$.  As $ r(L) = \lim_k \|L^k\|^{1/k}$, this implies that $\alpha \leq r(L) \leq \beta$. 
	\end{proof}
	
The following statement can be found in  \cite[Claim 3.1]{DS}, which contains an inaccuracy in its proof. To be precise, the assertion on \cite[page 80]{DS} that $g$ is an eigenvector of $L'$ seems unjustified. For completeness we give a  proof in the Appendix.   
	\begin{lemma}\label{Fcupb}
		If $F \subset \mathbb{N}$ is finite with $|F|\geq 2$, and  $\sigma= \dim_{\mathcal{H}}(J_{F})$, then  there exists a $C_F>1$ such that for all $n \in \mathbb{N}\setminus F$ we have that 
		\begin{equation}\label{Acupn} \sigma + C_F^{-1} n^{-2\sigma} \leq \dim_{\mathcal{H}}(J_{F\cup\{n\}}) \leq \sigma + C_F n^{-2\sigma}. 
		\end{equation} 
		Moreover, if $|F|=1$, then $\lim_{n\to\infty} \dim_{\mathcal{H}}(J_{F\cup\{n\}}) =0$.
		\end{lemma}

The following  result can be found in \cite{MU1}. 
\begin{theorem}\label{approx} Let  $F\subseteq \mathbb{N}$,  with $|F|=\infty$. If $F_1\subset F_2\subset \ldots \subset F$ with each $F_n$ finite and $\cup_n F_n=F$, then 
\[
\lim_n \mathrm{dim}_{\mathcal{H}}(J_{F_n}) = \mathrm{dim}_{\mathcal{H}}(J_F).
\]
\end{theorem}	

We will also need the following fact, see \cite[Proposition 2.7]{CLU2}. The same result can be found in \cite{N2} where different methods are used.
\begin{proposition}\label{prop:increasing}
If $A,B\subset \mathbb{N}$ and there exists a non-decreasing bijection $\tau\colon A\to B$, then \[\dim_{\mathcal{H}}(J_{B})\leq  \dim_{\mathcal{H}}(J_{A}).\]
\end{proposition}

In our arguments we occasionally need explicit upper and lower bounds for $\mathrm{dim}_{\mathcal{H}}(J_A)$ for specific finite sets $A\subset \mathbb{N}$. To get these bounds we used the rigorous numerical methods developed by Falk and Nussbaum in \cite{FN1,FN2} and the Matlab code from

\begin{center}
 	\url{https://sites.math.rutgers.edu/~falk/hausdorff/codes.html}  
\end{center}

The table below lists the bounds that are sufficient for our purposes, which were obtained by running the Matlab code with number of intervals $N=200$. It should, however, be noted that much sharper bounds can be obtained by using the numerical methods from \cite{FN2,FN3}. In some cases, for instance $A=\{1,2\}$, very sharp estimates exist, see e.g.,  \cite{FN3} and \cite{JP}.

\begin{table}[!h] 
\centering
\caption{Upper and lower bounds for Hausdorff dimension} \label{table}
\begin{small}
\begin{tabular}{r|rc||r|r} \toprule
$\{1,2\}$  & [0.531277, 0.531281]  & \qquad & $\{ 1, 2^{10} \}$ & [0.150819, 0.150820]     \\
$\{1, 3\}$ & [0.454487, 0.454490]  & &$\{ 1, 2^{11} \}$ & [0.140914, 0.140915]   \\
$\{ 1, 2^2\}$ &[0.411181, 0.411184] & &$\{1, 2, 4\}$ & [0.669217, 0.669223]    \\
$\{1, 2^3\}$ & [0.333644, 0.333646] & &$\{ 1, 2^5, 3^5\}$ & [0.272593, 0.272595]\\
$\{1, 2^4\}$ & [0.280974, 0.280976] & &$\{ 1, 2^6, 3^6 \}$ & [0.238624, 0.238626]\\
$\{1,2^5\}$ &[0.243375, 0.243377] & &$\{ 1, 2^7,3^7 \}$ & [0.212932, 0.212933]   \\
$\{ 1, 2^6 \}$ & [0.215370, 0.215371] &  &$\{ 1, 2^8,3^8 \}$ & [0.192784, 0.192786] \\
$\{ 1, 2^7 \}$ & [0.193748, 0.193749]  & & $\{ 1, 2^9,3^9 \}$ & [0.176528, 0.176529]  \\
$\{ 1, 2^8 \}$ & [0.176544, 0.176545] & & $\{ 1, 2^{10},3^{10} \}$ & [0.163106, 0.163107]  \\
$\{ 1, 2^9 \}$ & [0.162508, 0.162510] & &$\{1,3^5,\ldots,100^5\}$ & [0.243455, 0.243456]\\
\bottomrule
	\end{tabular}
    \end{small} 
\end{table}

To prove Theorems \ref{thm2} and \ref{Mq} we will need to consider Perron-Frobenius operators $L_{s,F}$ where $|F|=\infty$. In that case some care needs to be taken, as $L_{s,F}$ may not be defined for all values of $s>0$. Indeed, if $F =\{a_1,a_2,\ldots\}\subseteq \mathbb{N}$ with $a_1<a_2<\ldots$, then $L_{s,F}\colon C([0,1])\to C([0,1])$ given by, 
\[
(L_{s,F} f)(x) = \sum_{n=1}^\infty \left(\frac{1}{a_n+x}\right)^{2s} f\left(\frac{1}{a_n+x}\right)\mbox{\quad for $x\in [0,1]$},
\] 
is a bounded linear operator for  $s>\sigma_0$, where $\sigma_0=\inf\{\sigma>0\colon \sum_{n=1}^\infty a_n^{-2\sigma}<\infty\}$. 
In the case where $F\subseteq P_q^*$ with $q\geq 2$ we have that $\sigma_0=0$, and for $F\subseteq M_q$ with $q\geq 1$ we have that $\sigma_0 \leq (2q)^{-1}$. 
In \cite[Section 5]{NPVL} the relation between the spectral radius $r(L_{s,F})$ and $\dimh(J_F)$ was investigated for $|F|=\infty$. In fact, the more general setting of iterated function systems was considered there. We will use some of the  results from \cite{NPVL}. 
\begin{lemma}\label{infradius}(\cite[Lemma 5.4]{NPVL})
If $F\subseteq \mathbb{N}$ with $|F|=\infty$, then $s\mapsto r(L_{s,F})$ is continuous and strictly decreasing for $s>\sigma_0$. 
\end{lemma}
For $F\subseteq \mathbb{N}$ with $|F|=\infty$ let $\sigma_\infty =\inf\{s>0\colon r(L_{s,F})<1\}$. 
\begin{theorem}\label{infdim}(\cite[Theorem 5.11]{NPVL})
If $F\subseteq \mathbb{N}$ with $|F|=\infty$, then $\dimh(J_F)= \sigma_\infty$.  
\end{theorem}
The reason for  $\sigma_\infty$ to be defined in that way in \cite{NPVL} is due to the fact that for general iterated function systems there need not be an $s>\sigma_0$ for which $r(L_{s,F})=1$. This, however, will not be an issue here.  We should mention that although the derivative of the map $\theta_1\colon x\to (1+x)^{-1}$  satisfies $|\theta_1'(0)| =1$ the results from \cite[Section 5]{NPVL} can be used. Indeed, as explained in \cite[Example 5.12]{NPVL}, to prove the results mentioned above one can work with the operator $L^2_{s,F}$ and the maps $\theta_a\circ\theta_b$, where $\theta_a\colon x \mapsto (a+x)^{-1}$ for $a\in\mathbb{N}$, as they have the property that $|(\theta_a\circ\theta_b)'(x)|\leq 4^{-1}$ for all $x\in [0,1]$. 

\section{Strict break points}
The concept of a strict break point plays a central role in the analysis of the dimension spectrum. The idea goes back to the work by Kesseb\"ohmer and Zhu \cite[Theorem 2.2]{KZ}, and is also used in \cite{CLU2}. 
\begin{definition}
Let $A=\{a_1,a_2,\ldots\}\subseteq \mathbb{N}$ with $a_1<a_2<\ldots$. Given $F\subset A$ finite and $0<s<\mathrm{dim}_{\mathcal{H}}(J_A)$, we say that $a_k\in A$ is a {\em break point for $(F,s)$} if $a_k>\max F$ and 
\[
\mathrm{dim}_{\mathcal{H}}(J_F)< s \leq \mathrm{dim}_{\mathcal{H}}(J_{F\cup\{a_k\}}).
\]
If $(F,s)$ has a break point, then by Lemma \ref{Fcupb} there exists a break point $a_{k_0}\in A$ such that $\dimh(J_{F\cup\{a_{k_0}\}})\geq s$ and $\dimh(J_{F\cup\{a_{k_0+1}\}})<s$, which is called a {\em strict break point} for $(F,s)$. Note that if $a_{k_0}$ is a strict break point for $(F,s)$, then $k_0\geq 2$. 
\end{definition}

Strict break points can be used to show that an $s\in (0,1)$ is in the dimension spectrum of $A$.
\begin{lemma}\label{break point seq}
Let $A\subseteq \mathbb{N}$ be infinite and $F_1\subset F_2\subset \ldots \subset A$  be a nested sequences of finite subsets with $\max F_n<\max F_{n+1}$ for all $n\geq 1$. 
If $0<s<\dimh(J_A)$ and for each $n$ there exists a strict break point $a_{m_n}$ for $(F_n,s)$, then $s\in \mathrm{DS}(A)$. 
\end{lemma}
\begin{proof}
Let $\sigma_n = \dim_{\mathcal{H}}(J_{F_n})<s$ for $n\geq 1$, and let $\sigma =\dimh(J_{F_\infty})$, where $F_\infty=\cup_n F_n$. From Theorem \ref{approx} we know that  $\sigma_n\to \sigma$ as $n\to\infty$, and $\sigma\leq s$, as $\sigma_n<s$ for all $n$. To complete the proof we  show that $\sigma =s$. Suppose, by way of contradiction, that $\sigma <s$. 

For $n\geq 1$, let $G_n =F_n\cup\{a_{m_n}\}$, so $\dim_{\mathcal{H}}(J_{G_n})\geq s$ for each $n$. For $a,b\in \mathbb{N}$, the maps $\theta_a\colon x\mapsto \frac{1}{a+x}$ and $\theta_b\colon x\mapsto \frac{1}{b+x}$ satisfy
\[
(\theta_a\circ\theta_b)'(x) = (a(b+x)+1)^{-2}\mbox{\quad for $x\in [0,1]$.}
\]
So, 
\begin{equation}\label{1/4}
 \left((\theta_a\circ\theta_b)'(x)\right)^{s-\sigma_n} =  (a(b+x)+1)^{-2(s-\sigma_n)}\leq 2^{-2(s-\sigma)} = 4^{-(s-\sigma)}.
 \end{equation}
We know, see for instance \cite[Lemma 3.4]{NPVL}, that 
\[
(L^2_{s,F_n}f)(x) = \sum_{a,b\in F_n} ((\theta_a\circ\theta_b)'(x))^sf((\theta_a\circ\theta_b)(x))\mbox{\quad for $f\in C([0,1])$. }
\]

Now let $v_n\in C([0,1])$ be the strictly positive eigenvector of $L_{\sigma_n,F_n}$ with $L_{\sigma_n,F_n}v_n =v_n$. Then 
\begin{eqnarray*}
(L^2_{s,F_n}v_n)(x) & = & \sum_{a,b\in F_n} ((\theta_a\circ\theta_b)'(x))^sv_n((\theta_a\circ\theta_b)(x))\\
  & \leq & 4^{-(s-\sigma)}\sum_{a,b\in F_n} ((\theta_a\circ\theta_b)'(x))^{\sigma_n}v_n((\theta_a\circ\theta_b)(x))\\
   & = & 4^{-(s-\sigma)}L^2_{\sigma_n,F_n}v_n(x)\\
   & = & 4^{-(s-\sigma)}v_n(x),
\end{eqnarray*} 
hence $r(L^2_{s,F_n}) \leq 4^{-(s-\sigma)}$ by Lemma \ref{basic}. As $r(L_{s,F_n}) = \lim_k \|L_{s,F_n}^k\|^{1/k} $, we find that 
\begin{equation}\label{est1}
r(L_{s,F_n})= \lim_k  \left(\|L_{s,F_n}^{2k}\|^{1/k}\right)^{1/2} = r(L_{s,F_n}^2)^{1/2} \leq 2^{-(s-\sigma)}.
\end{equation}

We know from Theorem \ref{thm:FN} that there exists a strictly positive function $w_s\in C([0,1])$ such that $L_{s,F_n}w_s = r(L_{s,F_n})w_s$. Now using (\ref{est1}) and (\ref{ineqv_s}) we get that 
\begin{eqnarray*}
(L_{s,G_n}w_s)(x) & = & (L_{s,F_n}w_s)(x) +\left(\frac{1}{a_{m_n}+x}\right)^{2s}w_s\left(\frac{1}{a_{m_n}+x}\right)\\
  & \leq & 2^{-(s-\sigma)} w_s(x) +  \left(\frac{1}{a_{m_n}}\right)^{2s}e^{2s}w_s(x),
\end{eqnarray*}
hence $r(L_{s,G_n})\leq 2^{-(s-\sigma)}  + a_{m_n}^{-2s}e^{2s}$. As $\dim_\mathcal{H}(J_{G_n})\geq s$, we know that $r(L_{s,G_n})\geq 1$, which gives 
\[
1\leq r(L_{s,G_n})\leq 2^{-(s-\sigma)}  + a_{m_n}^{-2s}e^{2s}
\]
for $n\geq 1$. This is impossible, since $a_{m_n}\to\infty$ and $s-\sigma>0$. 
\end{proof}
The following lemma is similar to \cite[Theorem 2.2]{KZ}.
\begin{lemma}\label{break point}
Suppose $A\subseteq \mathbb{N}$ is infinite and $0<s<\mathrm{dim}_{\mathcal{H}}(J_A)$. If for each $F\subset A$ finite with strict break point $a_{k_0}\in A$ for $(F,s)$ we have that $ s < \mathrm{dim}_{\mathcal{H}}(J_{F\cup T})$, where $T=\{a_n\in A\colon n>k_0\}$,
then $s\in \mathrm{DS}(A)$.
\end{lemma}
\begin{proof} Let $A=\{a_1,a_2,\ldots\}\subseteq\mathbb{N}$ with $a_1<a_2<\ldots$. 
As $0<s<\dim_{\mathcal{H}}(J_A)$, it follows from Theorem \ref{approx} that there exists a $k_1\geq 1$ such that $F_1=\{a_1,\ldots,a_{k_1}\}$ satisfies 
\[
\dim_{\mathcal{H}}(J_{F_1})<s\mbox{\quad and \quad }\dim_{\mathcal{H}}(J_{F_1\cup\{a_{k_1+1}\}})\geq s.
\]
Now let $m_1\geq k_1+1$ be such that $a_{m_1}$ is a strict break point for $(F_1,s)$. It follows from the assumption that $\dim_{\mathcal{H}}(J_{F_1\cup T_1})\geq s$, where $T_1=\{a_k\in A\colon k>m_1\}$. 
In that case we can use  Theorem \ref{approx} again and find a  $k_2>m_1$ such that $F_2=F_1\cup\{a_{m_1+1},\ldots,a_{k_2}\}$ satisfies 
\[
\dim_{\mathcal{H}}(J_{F_2})<s\mbox{\quad and \quad }\dim_{\mathcal{H}}(J_{F_2\cup\{a_{k_2+1}\}})\geq s.
\]
Now let $m_2\geq k_2+1$ be such that $a_{m_2}$ is a strict break point for $(F_2,s)$. Thus, $\dim_{\mathcal{H}}(J_{F_2\cup T_2})\geq s$, where $T_2=\{a_k\in A\colon k>m_2\}$ by the assumption. 

Repeating this process, we find a nested sequence $F_1\subset F_2\subset \ldots \subset A$ of finite subsets, with $\max F_n<\max F_{n+1}$ for all $n$, and indices $m_1<m_2<\ldots$ such that $a_{m_n}\in A$ is a strict break point for $(F_n,s)$  for all $n$. It now follows from Lemma \ref{break point seq} that $s\in \mathrm{DS}(A)$.
\end{proof}

We will also need a general criterion to identify gaps in the dimension spectrum. This criterion is similar to the one given by Kesseb\"ohmer and Zhu in \cite[Theorem 2.4]{KZ}.  For completeness we include a proof of the statement we will need for our purposes.  To formulate it, we introduce some notation.  

Let $A=\{a_1,a_2,\ldots\}\subseteq \mathbb{N}$, with $a_1<a_2<\ldots$, $I_k =\{a_1,\ldots,a_k\}$, and  $T_k =\{a_{k+1},a_{k+2},\ldots\}$ for $k\geq 1$. Denote $\alpha_k = \dim_{\mathcal{H}}(J_{I_{k-1}\cup T_k})=\dim_{\mathcal{H}}(J_{A\setminus\{a_k\}})$ and $\beta_k = \dim_{\mathcal{H}}(J_{I_k})$ for $k\geq 1$. Here $I_0=\emptyset$. Given $F\subset A$ finite, we write 
\begin{equation}\label{Fhash}
F^\sharp = (F\setminus\max F) \cup\{a_n\in A\colon a_n>\max F\}.
\end{equation}

\begin{lemma}\label{nwdense} If $\alpha_k<\beta_k$ for some $k\geq 2$, and for each finite $F\subset A$ with $\beta_k <\dim_{\mathcal{H}}(J_F)<\alpha_{k+1}$ we have that 
\[
\dim_{\mathcal{H}}(J_{F^\sharp})<\dim_{\mathcal{H}}(J_F),
\]
then $\mathrm{DS}(A)$ is nowhere dense in $(\beta_k,\alpha_{k+1})$. 
\end{lemma}
\begin{proof}
Let $F\subset A$ finite with $\dim_{\mathcal{H}}(J_F)=s$ and $\beta_k<s<\alpha_{k+1}$. We claim that there exists no $G\subset A$ finite with $\dim_{\mathcal{H}}(J_G)\in  (\beta_k,\alpha_{k+1})$ such that 
\[
\dim_{\mathcal{H}}(J_{F^\sharp})<\dim_{\mathcal{H}}(J_G)<\dim_{\mathcal{H}}(J_F).
\]

Suppose that $G\subset A$ finite with $\dim_{\mathcal{H}}(J_G)\in  (\beta_k,\alpha_{k+1})$. Let $a_q =\min (G\cup F)\setminus (G\cap F)$. We note that $I_k\subseteq F,G$, since $\alpha_k<\beta_k \leq \dim_{\mathcal{H}}(J_F), \dim_{\mathcal{H}}(J_G)$ and the fact that 
$\dim_{\mathcal{H}}(J_{A\setminus\{a_k\}})\geq \dim_{\mathcal{H}}(J_{A\setminus\{a_m\}})$ for $m\leq k$ by Proposition \ref{prop:increasing} and Theorem \ref{approx}.  So, $q>k\geq 2$. 

There are four cases to consider. Firstly, $a_q =\max F$. In that case, $G \supseteq F\setminus\max F$, hence $G\subseteq F^\sharp$. 
As $\dim_{\mathcal{H}}(J_{F^\sharp})<\dim_{\mathcal{H}}(J_{F})$, we conclude that $\dim_{\mathcal{H}}(J_{G})\leq\dim_{\mathcal{H}}(J_{F^\sharp})$. 
 
 The second case to consider is $a_q>\max F$. In that case $F\subset G$, hence $\dim_{\mathcal{H}}(J_{F})\leq \dim_{\mathcal{H}}(J_{G})$.
 
 As a third case we suppose that $a_q<\max F$ and $a_q \in F$.  Let $F_* = F\cap\{a_1,\ldots,a_q\}\supset I_k$. Then $F_*\setminus \{a_q\} = F_*\setminus\max F_*$, so that $G\subset F_*^\sharp$ and $F_*\subseteq F\setminus \max F\subset F^\sharp$. As 
 \[
 \alpha_{k+1}>\dim_{\mathcal{H}}(J_F)\geq \dim_{\mathcal{H}}(J_{F_*})>\dim_{\mathcal{H}}(J_{I_k}) = \beta_k,
 \]
 it follows from the assumption that 
 \[
 \dim_{\mathcal{H}}(J_{G})\leq \dim_{\mathcal{H}}(J_{F^\sharp_*})<\dim_{\mathcal{H}}(J_{F_*})\leq \dim_{\mathcal{H}}(J_{F^\sharp}),
 \]
 which settles this case. 
 
 For the remaining case we need to consider $a_q<\max F$ and $a_q \in G$.  In that case we consider $G_* = G\cap \{a_1,\ldots,a_q\}\supset I_k$. 
 Then $F\subset G_*^\sharp$, and 
 \[
 \beta_{k}<\dim_{\mathcal{H}}(J_{G_*})\leq \dim_{\mathcal{H}}(J_{G})< \alpha_{k+1}.
 \]
 So, using the assumption we find that 
 \[
 \dim_{\mathcal{H}}(J_F)<\dim_{\mathcal{H}}(J_{G^\sharp_*})<\dim_{\mathcal{H}}(J_{G_*}) \leq \dim_{\mathcal{H}}(J_{G}),
 \]
which completes the proof of the claim.

It follows from the claim that any open interval $I\subseteq (\beta_k,\alpha_{k+1})$ contains an open interval $I_0$ such that $\mathrm{DS}(A)\cap I_0 $ is empty. Indeed, if $\mathrm{DS}(A)\cap I$ is non-empty, then there exists $B\subset A$ with $\dim_{\mathcal{H}}(J_B)\in I$. By Theorem \ref{approx} we know that there exists $F\subset B$ finite with $\dim_{\mathcal{H}}(J_F)\in I$. 
From the claim we know that there exists no $G\subset A$ finite with 
\[
\dim_{\mathcal{H}}(J_{F^\sharp})<\dim_{\mathcal{H}}(J_G)<\dim_{\mathcal{H}}(J_F). 
\]
So, if we put $I_0=(\dim_{\mathcal{H}}(J_{F^\sharp}), \dim_{\mathcal{H}}(J_{F}))$, then  $\mathrm{DS}(A)\cap I_0 $ is empty by Theorem \ref{approx}. This shows that  $\mathrm{DS}(A)$ is nowhere dense in $(\beta_k,\alpha_{k+1})$.
\end{proof}

\section{Bounds for  $\dim_{\mathcal{H}}(J_{\{1,n\}})$}
To establish the results we need a generic  lower bound for the Hausdorff dimension of $J_{\{1,n\}}$. The main idea is to use the positive eigenvector for the operator 
\[
(L_{s,\{1\}} f)(x) =\left(\frac{1}{1+x}\right)^{2s}f\left(\frac{1}{1+x}\right).
\]
\begin{lemma}\label{eigenvector}
Let $\mu>0$ and $s\geq 0$. The operator $L_{s,\{\mu\}}\colon C([0,\frac{1}{\mu}])\to C([0,\frac{1}{\mu}])$ given by 
\[
(L_{s,\{\mu\} }f)(x) =\left(\frac{1}{\mu+x}\right)^{2s}f\left( \frac{1}{\mu+x}\right)
\]
has $v_s(x) = \left(\frac{1}{\lambda +x}\right)^{2s}$, where 
\[
\lambda = \frac{\mu+\sqrt{\mu^2+4}}{2},
\] 
as a strictly  positive eigenvector with eigenvalue $\lambda^{-2s}$. In particular, $r(L_{s,\{\mu\}}) = \lambda^{-2s}$. 
\end{lemma}
\begin{proof} Note that $\lambda$ satisfies $\lambda^2 -\mu\lambda -1=0$, hence 
\[
v_s\left(\frac{1}{\mu+x}\right) = \left(\frac{1}{\lambda +\frac{1}{\mu +x}}\right)^{2s} = \left(\frac{\mu +x}{\mu\lambda +1 +\lambda x}\right)^{2s} 
=  \left(\frac{\mu +x}{\lambda^2 +\lambda x}\right)^{2s} = \lambda^{-2s}(\mu+x)^{2s}v_s(x).  
\]
This implies that $L_{s,\{\mu\}}v_s(x) = \lambda^{-2s}v_s(x)$. As $v_s$ is strictly positive,  $r(L_{s,\{\mu\}}) =\lambda^{-2s}$ by Lemma \ref{basic}.
\end{proof}
Using this results we now prove the following estimates for the Hausdorff dimension of $J_{\{1,n\}}$. 
\begin{theorem}\label{bounds} For $n\geq 1$ let 
\[
s_-(n) =\max\left\{s\geq 0\colon \lambda^{-2s}\left(1+\left(\frac{\lambda}{n+\lambda -1}\right)^{2s}\right)\geq 1\right\}
\]
and 
\[
s_+(n) =\min\left\{s\geq 0\colon \lambda^{-2s}\left(1+\left(\frac{\lambda+1}{n+\lambda}\right)^{2s}\right)\leq 1\right\},
\]
where $\lambda = \frac{1+\sqrt{5}}{2}$. 
Then 
\[
s_-(n)\leq  \dim_{\mathcal{H}}(J_{\{1,n\}}) \leq s_+(n).
\]
\end{theorem}
\begin{proof}
Note that if $v_s(x) =\left(\frac{1}{\lambda+x}\right)^{2s}$, so $L_{s,\{1\}}v_s = \lambda^{-2s}v_s$,  then 
\[
v_s\left(\frac{1}{x+n}\right) =  \left(\frac{1}{\lambda+\frac{1}{n+x}}\right)^{2s} =\left(\frac{n+x}{\lambda(n+x)+1}\right)^{2s} =\frac{(n+x)^{2s}}{\lambda^{2s}(n+x +\lambda^{-1})^{2s}} = \frac{(n+x)^{2s}}{\lambda^{2s}(n+x +\lambda -1)^{2s}}, 
\]
as $\lambda^{-1} =\lambda -1$. This implies that 
\[
(L_{s,\{1,n\}}v_s)(x) = \lambda^{-2s}\left( 1+\left(\frac{\lambda+x}{n+x +\lambda -1}\right)^{2s}\right)v_s(x).
\]
For $n>1$ and $x\in [0,1]$ the continuous function, 
\[
s\mapsto  \lambda^{-2s}\left( 1+\left(\frac{\lambda+x}{n+x +\lambda -1}\right)^{2s}\right),
\]
is strictly decreasing, positive, and at $s=0$ takes the value 2.  Moreover, for $n>1$ and $s>0$, the function 
\[
x\mapsto  \lambda^{-2s}\left( 1+\left(\frac{\lambda+x}{n+x +\lambda -1}\right)^{2s}\right)
\]
is strictly increasing on $[0,1]$. Thus, its maximum is $s_+(n)$, which is attained at $x=1$, and its minimum is $s_-(n)$, which is attained at $x=0$.

It follows that for $s\geq s_+(n)$ that $L_{s,\{1,n\}}v_s(x)\leq v_s(x)$, hence  $r(L_{s,\{1,n\}})\leq 1$ by Lemma \ref{basic}. So, by Theorem \ref{thm:FN} we get that $ \dim_{\mathcal{H}}(J_{\{1,n\}}) \leq s_+(n)$. Similarly, for $s\leq s_-(n)$ we have that $L_{s,\{1,n\}}v_s(x)\geq v_s(x)$, so that  $r(L_{s,\{1,n\}})\geq 1$. So, by Theorem \ref{thm:FN} we get that $ \dim_{\mathcal{H}}(J_{\{1,n\}}) \geq s_-(n)$.

\end{proof}
We can use the previous theorem to derive a general lower bound for $\dimh(J_{\{1,n\}})$ for $n\geq 4$. 
\begin{corollary}\label{lowerbound} For each $n\geq 4$ we have that
\[
 \dim_{\mathcal{H}}(J_{\{1,n\}})>\frac{0.52679}{\ln(n)}.
\]
\end{corollary}
\begin{proof}
We need to show for each integer $n\geq 4$ that $\frac{0.52679}{\ln(n)}<s_-(n)$. For $x\geq 4$ let 
\[
s(x) =\frac{c}{\ln x}\mbox{\quad and \quad } h(x) = \left( \frac{1}{\lambda}\right)^{2s(x)} + \left( \frac{1}{x+\lambda-1}\right)^{2s(x)}.
\]
Here $c> 0$ is a constant which will be chosen later to get the lower bound for $x\geq 4$. But for the moment it is useful to work with $c$ and any $x\geq 4$, because the method of proof gives a way to get a better constant if one has that $x\geq N$ for some fixed $N$. 

By Theorem \ref{bounds} we need to show that $h(x)>1$ for all $x\geq 4$. We first show that $h'(x)>0$ for all $x\geq 4$, and subsequently find a suitable constant $c>0$ such that $h(4)> 1$. Note that 
\[
s'(x)  = -\frac{c}{x\ln^2(x)}<0
\]
for $x\geq 4$, and 
\[
h'(x) = 2s'(x)\left(\frac{1}{x+\lambda -1}\right)^{2s(x)}\left( \left(\frac{x+\lambda -1}{\lambda}\right)^{2s(x)}\ln\left(\frac{1}{\lambda}\right) + \ln\left(\frac{1}{x+\lambda-1}\right) - \frac{s(x)}{(x+\lambda -1)s'(x)}\right).
\]
So, $2s'(x)\left(\frac{1}{x+\lambda -1}\right)^{s(x)}<0$ and  
$ \left(\frac{x+\lambda -1}{\lambda}\right)^{2s(x)}\ln\left(\frac{1}{\lambda}\right) <0$.
Moreover, $-\frac{s(x)}{s'(x)} =  x\ln (x)$, so that 
\[
- \frac{s(x)}{(x+\lambda -1)s'(x)} = \frac{x\ln (x)}{(x+\lambda -1)}< \frac{x\ln (x+\lambda -1)}{(x+\lambda -1)}.
\]
This implies that 
\[
\ln\left(\frac{1}{x+\lambda-1}\right) - \frac{s(x)}{(x+\lambda -1)s'(x)} < -\ln(x+\lambda -1)\left(1 - \frac{x}{x+\lambda -1} \right)< 0, 
\]
so $h'(x)>0$  for all $x\geq 4$.

For $x=4$ and $s(4) = 0.52679/\ln(4)$ a direct calculation shows that  
\[
h(4) = \left( \frac{1}{\lambda}\right)^{\frac{0.52679}{\ln 2}} + \left( \frac{1}{3+\lambda}\right)^{\frac{0.52679}{\ln 2}} >1.
\]
Thus, $\dimh(J_{\{1,n\}})> \frac{0.52679}{\ln(n)}$. 
\end{proof}

In particular, we find that $0.379998\leq \dimh(J_{\{1,4\}})$, which is a surprisingly good lower bound considering the estimates in Table \ref{table}. 

To establish Theorem \ref{Mq} we will also need a lower bound for $\dimh(J_{\{1,2^q\}})$ for $q\geq 12$. Using the same method as in the proof of Corollary \ref{lowerbound}  we need to find a constant $c>0$ such that for $x=2^{12}$ and $s(2^{12}) = \frac{c}{12\ln(2)}$ we have that 
\[
h(2^{12}) = \left( \frac{1}{\lambda}\right)^{\frac{c}{6\ln 2}} + \left( \frac{1}{2^{12}+\lambda-1}\right)^{\frac{c}{6\ln 2}} >1.
\]
In this case, one can check that $c =1.0571$ gives $h(2^{12})>1.005$, hence we have for $q\geq 12$ that 
\begin{equation}\label{est2^12}
\dimh(J_{\{1,2^q\}})\geq \frac{1.0571}{q\ln(2)} \geq \frac{1.525}{q}.
\end{equation}

\section{Proof of Theorem \ref{thm1}}
\begin{proof}[Proof of Theorem \ref{thm1}] Clearly $0$ and $\sigma = \dim_{\mathcal{H}}(J_A)$ are in the dimension spectrum of $A$. Take $0<s<\sigma$. We will use Lemma \ref{break point} to show that $s\in\mathrm{DS}(A)$. For $m\geq 1$ let $I_m = \{a_1,\ldots,a_m\}$ and let $u\in C([0,1])$ be the constant 1 function.  By Theorem \ref{approx} we know that $\sigma_m =\dim_{\mathcal{H}}(J_{I_m})\to \sigma$.  

Note that for each $m\geq 1$ and $x\in [0,1]$ we have that 
\[
(L_{s,I_m}u)(x)\leq \sum_{j=1}^m \left(\frac{1}{a_j}\right)^{2s}=:\alpha_m(s).
\]
We claim that $\alpha_m(s)>1$ for all $m$ sufficiently large. Indeed, if $\alpha_m(s)\leq 1$ for all $m$, then $r(L_{s,I_m})\leq 1$ for all $m\geq 1$ by Lemma \ref{basic}. 
As $0<s<\sigma$, we know from Theorem \ref{thm:FN} that 
\[
1=r(L_{\sigma_m,I_m}) < r(L_{s,I_m})\leq \alpha_m(s)\leq 1
\]
for all $m$ sufficiently large, since $\sigma_m>s$ for all $m$ large. This is impossible, hence $\alpha_m(s)>1$ for all $m$ sufficiently large. 

Now let $F\subset A$ finite and $a_{k_0}\in A$ be a strict break point for $(F,s)$. So, $r(L_{s,F\cup\{a_{k_0}\}})\geq 1$. Let $v_s$ be the strictly positive eigenvector for $L_{s,F\cup\{a_{k_0}\}}$, and set $H_m = F\cup\{a_{k_0+j}\colon j = 1,\ldots,m\}$. For $x\in [0,1]$, we have that 
\[
\frac{a_{k_0}+x}{a_{k_0+j}+x}\geq \frac{a_{k_0}}{a_{k_0+j}},
\]
so that 
\[
\left( \frac{1}{a_{k_0+j}+x}\right)^{2s} \geq \left( \frac{a_{k_0}}{a_{k_0+j}}\right)^{2s}\left( \frac{1}{a_{k_0}+x}\right)^{2s}\mbox{\quad for $j=1,\ldots,m$}.
\]

By Theorem \ref{thm:FN}, $v_s$ is a decreasing function on $[0,1]$. This implies that 
\begin{eqnarray*}
(L_{s,H_m}v_s)(x) & = & (L_{s,F}v_s)(x) + \sum_{j=1}^m \left( \frac{1}{a_{k_0+j}+x}\right)^{2s}v_s\left( \frac{1}{a_{k_0+j}+x}\right)\\
 & \geq & (L_{s,F}v_s)(x) +  \left( \frac{1}{a_{k_0}+x}\right)^{2s}v_s\left( \frac{1}{a_{k_0}+x}\right)\sum_{j=1}^m \left( \frac{a_{k_0}}{a_{k_0+j}}\right)^{2s}.
\end{eqnarray*}
As $k_0\geq 2$, we can use the assumption, $a_ma_n\geq a_{m+n}$ for all $m\geq n\geq 1$ with $m\geq 2$, to find that 
\[
\sum_{j=1}^m \left( \frac{a_{k_0}}{a_{k_0+j}}\right)^{2s} \geq \sum_{j=1}^m \left( \frac{1}{a_{j}}\right)^{2s} =\alpha_m(s).
\]
As $\alpha_m(s)>1$ for all $m\geq 1$ sufficiently large, there exists a constant $\lambda>1$ such that 
\[
(L_{s,H_m}v_s)(x) \geq (L_{s,F}v_s)(x) +\lambda \left( \frac{1}{a_{k_0}+x}\right)^{2s}v_s\left( \frac{1}{a_{k_0}+x}\right)
\]
for all $m$ large. Now using Lemma \ref{scaling} we conclude that there exists $\mu>1$ such that 
\[
(L_{s,H_m}v_s)(x) \geq \mu\left((L_{s,F}v_s)(x) +\left( \frac{1}{a_{k_0}+x}\right)^{2s}v_s\left( \frac{1}{a_{k_0}+x}\right)
\right)\geq \mu v_s(x)
\]
for all $m$ large. 
This implies that $r(L_{s,H_m})>1$ for all $m$ large by Lemma \ref{basic}, hence $\dim_{\mathcal{H}}(J_{H_m})>s$ for all $m$ large. As $F\cup T\supset H_m$, where $T =\{a_n\colon n>k_0\}$, we have that $\dim_{\mathcal{H}}(J_{F\cup T})>s$. The result now follows from Lemma \ref{break point}.
\end{proof}
\section{Gaps in $\mathrm{DS}(P^*_q)$: Proof of Theorem \ref{thm2}}
To establish the structure of the dimension spectrum for $P_n^*$, the following result is useful.
\begin{theorem}\label{En*} Suppose that $F\subset P_q^*$ is finite.  If $q\geq 3$ and $\{1,q\}\subseteq F$, or, $q=2$ and $\{1,2,4\}\subseteq F$, then  
\[
\dim_{\mathcal{H}}(J_{F^\sharp})<\dim_{\mathcal{H}}(J_{F}),
\]
where $F^\sharp$ is given by (\ref{Fhash}). 
\end{theorem}  
\begin{proof}
Suppose that  $F\subset P_q^*$ is finite with $\max F = q^k$. Set $G=F\setminus\max F$ and, for $0<s\leq 1$, let $v_s$ be the positive eigenvector of $L_{s,F}$ with eigenvalue $\lambda_s=r(L_{s,F})$. 

Then for each $m\geq q^k$ and $x\in[0,1]$ we have that $\frac{q^k+x}{m+x}\leq \frac{q^k+1}{m+1}$. Furthermore by (\ref{ineqv_s}), $v_s$ satisfies 
\[
v_s\left(\frac{1}{m+x}\right)\leq e^{2s\left(\frac{1}{q^k+x}-\frac{1}{m+x}\right)} v_s\left(\frac{1}{q^k+x}\right)\leq e^{\frac{2s}{q^k}} v_s\left(\frac{1}{q^k+x}\right). 
\]
Note that for $s>0$ the operator $L_{s,F^\sharp}$ is defined and bounded. Moreover, 
\begin{eqnarray*}
(L_{s,F^\sharp}v_s)(x)
& = & (L_{s, G}v_s)(x) + \sum_{j=1}^\infty \left(\frac{1}{q^{k+j}+x}\right)^{2s}v_s\left(\frac{1}{q^{k+j}+x}\right)\\
 & \leq & (L_{s, G}v_s)(x) +  \left(\frac{1}{q^k+x}\right)^{2s}v_s\left(\frac{1}{q^k+x}\right)e^{\frac{2s}{q^k}} 
 \sum_{j=1}^\infty \left(\frac{q^k+x}{q^{k+j}+x}\right)^{2s}.\\
 \end{eqnarray*}
 We have that 
 \[
 \sum_{j=1}^\infty \left(\frac{q^k+x}{q^{k+j}+x}\right)^{2s}  \leq   \sum_{j=1}^\infty \left(\frac{q^k+1}{q^{k+j}+1}\right)^{2s}
  \leq  \left(\frac{q^k+1}{q^{k}}\right)^{2s} \sum_{j=1}^\infty \left(\frac{1}{q^{j}}\right)^{2s}=  \frac{\left(1+\frac{1}{q^k}\right)^{2s}}{q^{2s} -1}.
\]

Now let 
\[
\gamma(k,q,s) =  \frac{\left(e^{\frac{1}{q^k}}\left(1+\frac{1}{q^k}\right)\right)^{2s}}{q^{2s} -1}\leq \frac{e^{\frac{4s}{q^k}}}{q^{2s} -1},
\]
as $e^x\geq 1+x$. Note that if $\gamma(k,q,s)<1$, then there exists by Lemma \ref{scaling} a $\mu<1$ such that 
$L_{s,F^\sharp}v_s\leq \mu L_{s, F}v_s  = \mu\lambda_s v_s$. 
In particular, if this holds for $s=\dim_{\mathcal{H}}(J_F)$, we get that  $L_{s,F^\sharp}v_s\leq   \mu v_s$. This would imply that $r(L_{s,F^\sharp})\leq \mu<1$, hence $\dim_{\mathcal{H}}(J_{F^\sharp})<s$ by Lemma \ref{infradius} and Theorem \ref{infdim}.  So we need to show that $\gamma(k,q,s_0)<1$ for $s_0=\dim_{\mathcal{H}}(J_F)$.  

Firstly suppose that $q\geq 4$ and $k= 1$, so $F=\{1,q\}$. By Corollary \ref{lowerbound}, $\frac{0.52679}{\ln(q)}< s_0\leq 1/2$, so that 
\[
\gamma(1,q,s_0) \leq \frac{e^{4s_0/q}}{q^{2s_0} -1}\leq \frac{e^{2/q}}{q^{2s_0} -1}<1, 
\] 
as $q^{\frac{1.05358}{\ln(q)}}-1  = e^{1.05358} -1>e^{0.5}\geq e^{2/q}$ for $q\geq 4$. 

Likewise, if $q\geq 4$ and $k\geq 2$, then $\frac{0.52679}{\ln(q)} \leq \dim_{\mathcal{H}}(J_{\{1,q\}})\leq s_0=\dim_{\mathcal{H}}(J_{F})\leq 1$ and $q^k\geq 2q$, so that 
\[
\gamma(k,q,s_0) \leq \frac{e^{4s_0/q^k}}{q^{2s_0} -1}\leq \frac{e^{2/q}}{q^{2s_0} -1}<1.
\] 

Let us now consider the case $q=3$ and $k\geq 2$. In that case
\[
\gamma(k,3,s_0) \leq \frac{e^{4s_0/3^k}}{3^{2s_0} -1}\leq \frac{e^{4/9}}{3^{2s_0} -1}<0.92<1, 
\]
since $s_0=\dim_{\mathcal{H}}(J_{\{1,3\}})\geq 0.454$, see Table \ref{table}.  

The case $q=3$ and $k=1$ requires a more refined estimate than $\gamma(1,3,s_0)$. In that case we have that  
\begin{eqnarray*}
(L_{s,F^\sharp}v_s)(x)
& = & (L_{s, G}v_s)(x) + \sum_{j=1}^\infty \left(\frac{1}{3^{1+j}+x}\right)^{2s}v_s\left(\frac{1}{3^{1+j}+x}\right)\\
 & \leq & (L_{s, G}v_s)(x) +  \left(\frac{1}{3+x}\right)^{2s}v_s\left(\frac{1}{3+x}\right)\sum_{j=1}^\infty \left(\frac{4}{3^{j+1}+1}\right)^{2s}e^{2s\left(\frac{1}{3} - \frac{1}{3^{j+1}}\right)}.
 \end{eqnarray*} 
Note that 
\begin{eqnarray*}
\sum_{j=1}^\infty \left(\frac{4}{3^{j+1}+1}\right)^{2s}e^{2s\left(\frac{1}{3} - \frac{1}{3^{j+1}}\right)}& \leq & 
4^{2s}\left(\left(\frac{e^{2/9}}{10}\right)^{2s} +\left(\frac{e^{8/27}}{28}\right)^{2s} + e^{2s/3}\sum_{j=3}^\infty \left(\frac{1}{3^{j+1}}\right)^{2s}\right) \\
 & = & 4^{2s}\left(\left(\frac{e^{2/9}}{10}\right)^{2s}  +\left(\frac{e^{8/27}}{28}\right)^{2s} +  \left(\frac{e^{1/3}}{27}\right)^{2s} \left(\frac{1}{3^{2s}-1}\right)\right). \\
\end{eqnarray*}
Now using the  fact that $0.454\leq s_0=\dim_{\mathcal{H}}(J_{\{1,3\}})\leq 0.455$, we get that 
\[
4^{2s}\left(\left(\frac{e^{2/9}}{10}\right)^{2s}  +\left(\frac{e^{8/27}}{28}\right)^{2s} +  \left(\frac{e^{1/3}}{27}\right)^{2s} \left(\frac{1}{3^{2s}-1}\right)\right)< 0.899<1,
\]
which gives the desired inequality.

Finally let us consider the case $q=2$ and $\{1,2,4\}\subseteq F$. 
If $k\geq 3$, then 
\[
\gamma(k,2,s_0) \leq \frac{e^{4s_0/2^k}}{2^{2s_0} -1}\leq  \frac{e^{s_0/2}}{2^{2s_0} -1}<0.915<1, 
\]
since $0.669\leq s_0=\dim_{\mathcal{H}}(J_{\{1,2,4\}})\leq 0.67$, see Table \ref{table}.  

If $k=2$, then $F=\{1,2,4\}$ and  $G=\{1,2\}$, so that 
\begin{eqnarray*}
(L_{s,F^\sharp}v_s)(x)
& = & (L_{s, G}v_s)(x) + \sum_{j=1}^\infty \left(\frac{1}{2^{2+j}+x}\right)^{2s}v_s\left(\frac{1}{2^{2+j}+x}\right)\\
 & \leq & (L_{s, G}v_s)(x) +  \left(\frac{1}{4+x}\right)^{2s}v_s\left(\frac{1}{4+x}\right)\sum_{j=1}^\infty \left(\frac{5}{2^{j+2}+1}\right)^{2s}e^{2s\left(\frac{1}{4} - \frac{1}{2^{j+2}}\right)}.
 \end{eqnarray*} 
Note that 
\begin{eqnarray*}
\sum_{j=1}^\infty \left(\frac{5}{2^{j+2}+1}\right)^{2s}e^{2s\left(\frac{1}{4} - \frac{1}{2^{j+2}}\right)} & \leq & 
5^{2s}\left(\left(\frac{e^{1/8}}{9}\right)^{2s} +\left(\frac{e^{3/16}}{17}\right)^{2s} + e^{s/2}\sum_{j=3}^\infty \left(\frac{1}{2^{j+2}}\right)^{2s}\right) \\
 & = & 5^{2s}\left(\left(\frac{e^{1/8}}{9}\right)^{2s}  +\left(\frac{e^{3/16}}{17}\right)^{2s} +  \left(\frac{e^{1/4}}{16}\right)^{2s} \left(\frac{1}{2^{2s}-1}\right)\right). \\
\end{eqnarray*}
Now using the  fact that $0.669\leq s_0=\dim_{\mathcal{H}}(J_{\{1,2,4\}})\leq 0.67$, we get that 
\[
5^{2s}\left(\left(\frac{e^{1/8}}{9}\right)^{2s}  +\left(\frac{e^{3/16}}{17}\right)^{2s} +  \left(\frac{e^{1/4}}{16}\right)^{2s} \left(\frac{1}{2^{2s}-1}\right)\right)< 0.984<1,
\]
which gives the desired inequality. 
\end{proof} 

Using the previous theorem it is now easy to prove Theorem \ref{thm2}. 
\begin{proof}[Proof of Theorem \ref{thm2}]
Suppose that  $q\geq 3$ and $k\geq 1$. To prove assertion (i) we first note that we can take $F= I_k=\{1,\ldots,q^k\}$ in Theorem \ref{En*} and conclude that $\mu^k<\nu^k$. To see that $(\mu^k,\nu^k)\cap \mathrm{DS}(P^*_q)=\emptyset$ we argue by contradiction. So, suppose that  $F\subseteq P^*_q$ is such that $\mu^k<\dimh(J_F)<\nu^k$. We claim that $\{1,\ldots,q^{k-1}\}\subset F$, as otherwise $F\subseteq P^*_q\setminus\{q^m\}$ for some $m\leq k-1$. In that case we get that $\dimh(J_F)\leq \mu^m\leq \mu^k$ by Proposition \ref{prop:increasing}, which is impossible. As  $\{1,\ldots,q^{k-1}\}\subset F$ and $\dimh(J_F)<\nu^k$, we know that $q^k\not\in F$. Thus, $F\subseteq P^*_q\setminus\{q^k\}$, which contradicts the fact that $\mu^k<\dimh(J_F)$. 

To prove assertion (ii) let $F\subset P^*_q$ be finite with $\nu^k<\dimh(J_F)<\mu^{k+1}$. Then $\{1,\ldots,q^k\}\subset F$, as otherwise $F\subset P^*_q\setminus\{q^m\}$ for some $m\leq k$, which would imply that $\dimh(J_F)\leq \mu^m<\nu^m\leq \nu^k$. As $\mu^k<\nu^k$ for all $k\geq 1$, we can combine Lemma \ref{nwdense} and Theorem \ref{En*} and conclude that $\mathrm{DS}(P^*_q)$ is nowhere dense in $(\nu^k,\mu^{k+1})$ for $k\geq 1$.  (Note that  $\alpha_k=\mu^{k-1}$ and $\beta_k =\nu^{k-1}$ for $k\geq 1$.) 

The proof for $q=2$ can be derived in the same way from Theorem \ref{En*} and Lemma \ref{nwdense}.

\end{proof}
\section{Proof of Theorem \ref{thm3}}	

\begin{proof}[Proof of Theorem \ref{thm3}]
Let $0< s< \frac{\ln 2}{2\ln q}$. To show that $s$ is in the dimension spectrum we verify the condition in Lemma \ref{break point}.
So, suppose that $F\subset P_q^*$ is finite with strict break point say $q^{k_0}$ for $(F,s)$. Let $v_s$ be the strictly positive eigenvector of $L_{s,F\cup \{q^{k_0}\}}$ with eigenvalue $\lambda_s = r(L_{s,F\cup \{q^{k_0}\}})\geq 1$ and let $T=\{q^k\colon k>k_0\}$. Set $T_m =\{q^{k_0+j}\colon 1\leq j\leq m\}$. 

We know from Theorem \ref{thm:FN} that $v_s$ is decreasing on $[0,1]$. Using this fact we find that for $x\in [0,1]$, 
\begin{eqnarray*}
L_{s,F\cup T_m} v_s)(x) & = & (L_{s,F}v_s)(x) + \sum_{j=1}^m \left(\frac{1}{q^{k_0+j}+x}\right)^{2s} v_s\left(\frac{1}{q^{k_0+j}+x}\right)\\
 & \geq & (L_{s,F}v_s)(x) + \left(\frac{1}{q^{k_0}+x}\right)^{2s} v_s\left(\frac{1}{q^{k_0}+x}\right) \sum_{j=1}^m \left(\frac{q^{k_0}+x}{q^{k_0+j}+x}\right)^{2s} \\
 & \geq & (L_{s,F}v_s)(x) + \left(\frac{1}{q^{k_0}+x}\right)^{2s} v_s\left(\frac{1}{q^{k_0}+x}\right) \sum_{j=1}^m \left(\frac{1}{q^{j}}\right)^{2s}.
\end{eqnarray*}
As $s<\frac{\ln 2}{2\ln q}$, we know that $\frac{1}{q^{2s}-1}>1$, hence there exists an $M$ such that $ \sum_{j=1}^{M} \left(\frac{1}{q^{j}}\right)^{2s}>1$.  So,  there exists a $\lambda >1$ such that for $x\in [0,1]$,
\[
(L_{s,F\cup T_{M}} v_s)(x) > (L_{s,F}v_s)(x) + \lambda \left(\frac{1}{q^{k_0}+x}\right)^{2s} v_s\left(\frac{1}{q^{k_0}+x}\right). 
\]
Now using Lemma \ref{scaling} we conclude that there exists $\mu>1$ such that $L_{s,F\cup T_{M}} v_s \geq \mu \lambda_sv_s(x)\geq  \mu v_s$, hence 
$r(L_{s,F\cup T_{M}}) \geq \mu>1$ by Lemma \ref{basic}. This implies that $\dim_{\mathcal{H}}(J_{F\cup T})\geq \dim_{\mathcal{H}}(J_{F\cup T_{M}})> s$ by Theorem \ref{thm:FN}. 

To complete the proof note that clearly $0$ is in the dimensions spectrum, but also $\frac{\ln 2}{2\ln q}$, as the dimensions spectrum is closed by \cite[Theorem 1.2]{CLU1}. 
\end{proof}

\section{The dimension spectrum of $M_q$: proof of Theorem \ref{Mq}}
We will first prove the final statement in  Theorem \ref{Mq}. In fact, we will show that the following  general condition on $A\subseteq \mathbb{N}$  implies that its dimension spectrum is a finite  union of disjoint nontrivial closed intervals. 

\begin{definition} Given an infinite set $A=\{a_1,a_2,\ldots\}\subseteq \mathbb{N}$ with $a_1<a_2<\ldots$, we say that $A$ has a {\em critical break point value} $k^*$  if for each $t\in\mathrm{DS}(A)$ with $0<t<\dimh(J_A)$ and each finite set $F\subset A$ with a strict break point $a_m$ for $(F,t)$ and $m>k^*$ we have that 
\[
\dimh(J_{F\cup\{a_n\colon n>m\}})>t.
\]
\end{definition}

\begin{proposition}\label{critical}
If  $A=\{a_1,a_2,\ldots\}\subseteq \mathbb{N}$, with $a_1<a_2<\ldots$, has a critical break point value, then for each $s\in\mathrm{DS}(A)$ there exists a $\delta>0$ such that 
$[s-\delta,s]\subseteq \mathrm{DS}(A)$ or $[s,s+\delta]\subseteq \mathrm{DS}(A)$. 
\end{proposition}
\begin{proof}
Let $s\in \mathrm{DS}(A)$ and  $F\subseteq A$ with $\dimh(J_F) = s$.  Suppose first that $F$ is finite. Take $m>k^*$ such that $a_m>\max F$, where $k^*$ is the critical break point value for $A$.  Set $t_1=\dimh(J_{F\cup\{a_k\colon k\geq m\}})>s$. We will show that each  $s<t<t_1$ is in $\mathrm{DS}(A)$.  As $t_1>s$, we know from Theorem \ref{approx} that  either $\dimh(J_{F\cup\{a_m\}})\geq t$, in which case we set $F_1=F$,  or, there exists  a $k_1\geq m$ such that $F_1=F \cup\{a_{m},\ldots, a_{k_1}\}$ satisfies 
\[
\dimh(J_{F_1})<t\mbox{\qquad and\qquad }\dimh(J_{F_1\cup\{a_{k_1+1}\}})\geq t.
\]
In both cases we find that $(F_1,t)$ has a strict break point, say $a_{m_1}$, with $m_1\geq m$.  Now using  that  $m_1>k^*$ we see that $\dimh(J_{F_1\cup\{a_k\colon k>m_1\}}) >t$. It again follows from Theorem \ref{approx} that there exists a $k_2>m_1$ such that $F_2=F_1 \cup\{a_{m_1+1},\ldots, a_{k_2}\}$ satisfies 
\[
\dimh(J_{F_2})<t\mbox{\qquad and\qquad }\dimh(J_{F_2\cup\{a_{k_2+1}\}})\geq t.
\]
Let $a_{m_2}$ be a strict break point for $(F_2,t)$.  Again, as $m_2>k^*$, we have that $\dimh(J_{F_2\cup\{a_k\colon k>m_2\}})>t$. Thus, by Theorem \ref{approx}  there exists a $k_3>m_2$ such that $F_3=F_2 \cup\{a_{m_2+1},\ldots, a_{k_3}\}$ satisfies 
\[
\dimh(J_{F_3})<t\mbox{\qquad and\qquad }\dimh(J_{F_3\cup\{a_{k_3+1}\}})\geq t.
\]
Let $a_{m_3}$ be a strict break point for $(F_3,t)$. By repeating this process we find a nested sequence of sets $F_1\subset F_2\subset \ldots\subset A$ with $\max F_n<\max F_{n+1}$  and strict break points $a_{m_n}$ for $(F_n,t)$ for each $n$. It now follows from Lemma \ref{break point seq} that $t\in\mathrm{DS}(A)$. 

In the case where $F$ is infinite we take $m>k^*$ such that  $F' = \{a_k\in F\colon k< m\}$ is nonempty, so $s_0:=\dimh(J_{F'})<s$. Set  $s_1=\dimh(J_{F'\cup\{a_k\colon k\geq m\}})\geq s$.  Then using exactly the same reasoning as in the first case  with $F'$ instead of $F$ it can be  shown that each $s_0<t<s_1$ is in $\mathrm{DS}(A)$.
\end{proof}

\begin{theorem}\label{intervals}
If  $A=\{a_1,a_2,\ldots\}\subseteq \mathbb{N}$ with $a_1<a_2<\ldots$ has a critical break point value, then $\mathrm{DS}(A)$ is the disjoint union of finitely many nontrivial closed intervals. 
\end{theorem}
\begin{proof} We know from Proposition \ref{critical} that each connected component of $\mathrm{DS}(A)$ is a closed nontrivial interval, as $\mathrm{DS}(A)$ is closed, see \cite[Theorem 1.2]{CLU1}.   It remains to show that it only has  finitely many connected components. Suppose by way of contradiction that it consists of infinitely many connected components, say $[\alpha_i,\beta_i]$ for $i\in I$. Let $F_i\subset A$ be such that $\dimh(J_{F_i}) = \alpha_i$. Note that for $\alpha_0:=0$ and $|F_0|=1$. For each other $i\in I$ we have that $|F_i|\geq 2$. 

As there are infinitely many $F_i$'s we know there exists an $F_j$ containing $a_{j_1}<a_{j_2}$  with  $j_2>k^*$, where $k^*$ is the critical break point value of $A$. 
Now let $F= F_j\cap \{a_k\colon k<j_2\}$ and set $s_0 = \dimh(J_F)<\alpha_j$ and $s_1=\dimh(J_{F\cup\{a_n\colon n\geq j_2\}})\geq \alpha_j$. To get the contradiction we now use the same argument as in the proof of Proposition \ref{critical} to show that each $s_0<t<\alpha_j$ is in $\mathrm{DS}(A)$. 

As $s_0<t<\alpha_j\leq s_1$, we know from Theorem \ref{approx} that  either $\dimh(J_{F\cup\{a_{j_2}\}})\geq t$, in which case we set $F_1=F$,  or, there exists  a $k_1\geq j_2$ such that $F_1=F \cup\{a_{j_2},\ldots, a_{k_1}\}$ satisfies 
\[
\dimh(J_{F_1})<t\mbox{\qquad and\qquad }\dimh(J_{F_1\cup\{a_{k_1+1}\}})\geq t.
\]
In both cases we find that $(F_1,t)$ has a strict break point, say $a_{m_1}$, with $m_1\geq j_2$.  As $m_1>k^*$, we know that $\dimh(J_{F_1\cup\{a_k\colon k>m_1\}}) >t$. It now follows from Theorem \ref{approx} that there exists a $k_2>m_1$ such that $F_2=F_1 \cup\{a_{m_1+1},\ldots, a_{k_2}\}$ satisfies 
\[
\dimh(J_{F_2})<t\mbox{\qquad and\qquad }\dimh(J_{F_2\cup\{a_{k_2+1}\}})\geq t.
\]
Let $a_{m_2}$ be a strict break point for $(F_2,t)$.  Iteratively repeating this process yields a nested sequence of sets $F_1\subset F_2\subset \ldots\subset A$ with $\max F_n<\max F_{n+1}$  and a strict break points $a_{m_n}$ for $(F_n,t)$ for each $n$. It now follows from Lemma \ref{break point seq} that $t\in\mathrm{DS}(A)$, which contradicts the fact that $[\alpha_j,\beta_j]$ is a connected component of $\mathrm{DS}(A)$.
\end{proof}
We will see that $M_q$ has a critical break point value for $q\geq 11$, namely $k^*=2q$. To show this we need an upper bound for $\dimh(J_{M_q})$ for $q\geq 11$. The following bound, which is not very sharp,  will be sufficient for our purposes. 
\begin{lemma}\label{upperMq} For $q\geq 11$ we have that $\dimh(J_{M_q})\leq \frac{2}{\sqrt{q}}$. 
\end{lemma}
\begin{proof}
Let $q\geq 11$ and $\frac{1}{2q}<s\leq \frac{2}{\sqrt{q}}$. For $k>2$ set $M_q^k= \{1,2^q,\ldots, k^q\}$. Using the positive eigenvector $v_s$ of $L_{s,\{1\}}$ with eigenvalue $\lambda^{-2s}$, where $\lambda = (1+\sqrt{5})/2$ from Lemma \ref{eigenvector}, we find that 
\[
L_{s,M_q^k}v_s(x) = \lambda^{-2s}\left( 1+\sum_{n=2}^k \left(\frac{\lambda+x}{n^q +x+\lambda -1}\right)^{2s}\right)v_s(x) \leq  \lambda^{-2s}\left( 1+\sum_{n=2}^k \left(\frac{\lambda+1}{n^q +\lambda }\right)^{2s}\right)v_s(x).
\]
As $\lambda^{-1}+1= \lambda$,
\begin{equation}\label{mu(s)}
\begin{split}
\lambda^{-2s}\left( 1+\sum_{n=2}^k \left(\frac{\lambda+1}{n^q +\lambda }\right)^{2s}\right) =  \lambda^{-2s}+\sum_{n=2}^k \left(\frac{\lambda}{n^q +\lambda }\right)^{2s}
\leq 
\lambda^{-2s}+\left(\frac{\lambda}{2^q}\right)^{2s} + \lambda^{2s}\int_{2}^\infty \frac{1}{x^{2qs}}\mathrm{d}x 
\\= \lambda^{-2s}+\left(\frac{\lambda}{2^q}\right)^{2s}\left( 1+\frac{2}{2qs-1}\right)=:\mu(s).
\end{split}
\end{equation}
Our goal is to show that $\mu(s)<1$ for $s=2/\sqrt{q}$ and $q\geq 11$. To establish this inequality set  
\[
h(x) := \lambda^{-4/\sqrt{x}}+\left(\frac{\lambda}{2^x}\right)^{4/\sqrt{x}}\left( 1+\frac{2}{4\sqrt{x}-1}\right)
\]
for $x\geq 1$. We need to show that $h(x)<1$ for all $x\geq 11$. Since $h(x)\to 1$ as $x\to\infty$, it suffices to show that $h$ is strictly increasing for $x\geq 11$. 

A direct computation gives
\[
h'(x) = \frac{2\ln \lambda}{x\sqrt{x}} \lambda^{-4/\sqrt{x}} + \left(\frac{\lambda}{2^x}\right)^{4/\sqrt{x}}\left[\left(1+\frac{2}{4\sqrt{x} -1}\right)\left( -\frac{2\ln 2}{\sqrt{x}} -\frac{2\ln \lambda}{x\sqrt{x}}\right) - \frac{4}{\sqrt{x}(4\sqrt{x}-1)^2}\right].
\]
To prove that $h'(x)>0$ for all $x\geq 11$, we show that
\[
 \frac{x\sqrt{x}}{2\ln \lambda} \left(\frac{2^x}{\lambda}\right)^{4/\sqrt{x}}h'(x)>0\mbox{\qquad for }x\geq 11.
\]
Note that 
\begin{equation*}
\begin{split}
 \frac{x\sqrt{x}}{2\ln \lambda} \left(\frac{2^x}{\lambda}\right)^{4/\sqrt{x}}h'(x) =   \lambda^{-8/\sqrt{x}}2^{4\sqrt{x}} - 
 \left[\left(1+\frac{2}{4\sqrt{x} -1}\right)\left( \frac{\ln 2}{\ln \lambda}x +1\right) +\frac{2x}{\ln \lambda(4\sqrt{x}-1)^2}\right]\\
 \geq  \lambda^{-8/\sqrt{11}}2^{4\sqrt{x}} - \left[\left(1+\frac{2}{4\sqrt{11} -1}\right)\left( \frac{\ln 2}{\ln \lambda}x +1\right) +\frac{2}{\ln \lambda(16-8/\sqrt{11})}\right]=:g(x),
\end{split}
\end{equation*}
and  a direct computation gives $g(11)>0$. Using the derivative of $g$ it is straightforward to verify that $g$ is an increasing function for $x\geq 11$, so $h'(x)>0$ for all  
$x\geq 11$. 

Thus, if we take $s=2/\sqrt{q}$ for $q\geq 11$ in (\ref{mu(s)}) we find that $\mu(s)<1$. This implies that $r(L_{s,M_q^k})\leq \mu(s)<1$, hence $\mathrm{dim}_H(J_{M_q^k})<s$ for all $k$ and $q\geq 11$ by Theorem \ref{thm:FN}. It now follows from Theorem \ref{approx} that $\dimh(J_{M_q})\leq s$ for $s=2/\sqrt{q}$.  
\end{proof}

Let us now show that $M_q$ has a critical break point value  for $q\geq 11$. 
\begin{theorem}\label{Mqcritical}
The set $M_q$ has a critical break point value $k^*=2q$ for $q\geq 11$. 
\end{theorem}
\begin{proof}
Suppose that $s\in\mathrm{DS}(A)$ with $0<s<\dimh(J_A)$ and $q\geq 11$. Let $k_0^q$ be a strict break point for    $(F,s)$, where  $F\subset A$ is a  finite set and  $k_0>2q$.
Let $H_m=F\cup\{k^q\colon k_0<k\leq m\}$ for $m>k_0$.  Consider the operator $L_{s,F\cup\{k_0^q\}}$ with positive eigenvector $v_s$ and eigenvalue $r( L_{s,F\cup\{k_0^q\}})\geq 1$, as $k_0^q$ is a strict break point for $(F,s)$. Then 
\begin{eqnarray*}
L_{s,H_m}v_s(x) & = & L_{s,F}v_s(x) + \sum_{k=k_0+1}^m \left(\frac{1}{k^q+x}\right)^{2s} v_s\left(\frac{1}{k^q+x}\right)\\
& \geq  &
L_{s,F}v_s(x) +  \left(\frac{1}{k_0^q+x}\right)^{2s} v_s\left(\frac{1}{k_0^q+x}\right)\sum_{k=k_0+1}^m\left(\frac{k_0^q}{k^q}\right)^{2s},
\end{eqnarray*}
as $v_s$ is decreasing by Theorem \ref{thm:FN} and $\frac{k_0^q+x}{k^q+x}\geq \frac{k_0^q}{k^q}$ for all $x\in [0,1]$ and $k\geq k_0$.  

We will now show that $\sum_{k=k_0+1}^m \left(\frac{k_0^q}{k^q}\right)^{2s}>1$ for all $m$ sufficiently large. Note that 
\[
\sum_{k=k_0+1}^\infty\left(\frac{k_0^q}{k^q}\right)^{2s} \geq k_0^{2qs}\int_{k_0+1}^\infty x^{-2qs}\mathrm{d}x = \left(\frac{k_0}{k_0+1}\right)^{2qs}\frac{k_0+1}{2qs-1}.
\]
The right-hand side is an increasing function in $k_0$. So, as $k_0>2q$ and $s<\dimh(J_{M_q})\leq 2/\sqrt{q}$ by Lemma \ref{upperMq}, we find that 
 \[
\sum_{k=k_0+1}^\infty\left(\frac{k_0^q}{k^q}\right)^{2s}\geq \left(\frac{2q+1}{2q+2}\right)^{2qs}\frac{2q+2}{2qs-1}\geq  \left(\frac{2q+1}{2q+2}\right)^{4\sqrt{q}}\frac{2q+2}{4\sqrt{q}-1}=:\tau(q).
\]
We will show that $\tau(q)>1$ for all $q\geq 11$. 
Note that the function $g(x)=  \left(\frac{2x+1}{2x+2}\right)^{4\sqrt{x}}$ has derivative 
\[
g'(x) = \left(\frac{2x+1}{2x+2}\right)^{4\sqrt{x}}\left(\frac{2}{\sqrt{x}}\left(\ln\left(\frac{2x+1}{2x+2}\right) \right) +4\sqrt{x}\left(\frac{2}{2x+1}-\frac{2}{2x+2}\right)\right).
\]
As
\begin{eqnarray*}
 \frac{2}{\sqrt{x}}\left(\ln\left(\frac{2x+1}{2x+2}\right) \right) +4\sqrt{x}\left(\frac{2}{2x+1}-\frac{2}{2x+2}\right)
        & \geq &   -\frac{2}{\sqrt{x}}\left(\frac{1}{2x+1}\right) +4\sqrt{x}\left(\frac{2}{2x+1}-\frac{2}{2x+2}\right)\\
        & = &   \left(\frac{-2\sqrt{x}}{x(2x+1)}\right) +\left(\frac{8\sqrt{x}}{(2x+1)(2x+2)}\right)\\
        & = &   \left(\frac{-2\sqrt{x}}{x(2x+1)}\right) +\left(\frac{2\sqrt{x}}{(2x+1)(x/2+1/2)}\right)\geq 0\\
\end{eqnarray*}
for $x\geq 1$, $g'$ is nonnegative, hence $g$  is an increasing function. Also the function $x\mapsto \frac{2x+2}{4\sqrt{x}-1}$ is increasing for $x\geq 11$. Thus $\tau(q)\geq \tau(11)\geq 1.112$ for all $q\geq 11$. 

It follows that for all $m$ sufficiently large that 
\[
\sum_{k=k_0+1}^m\left(\frac{k_0^q}{k^q}\right)^{2s}>1.
\]
Thus there exists a $\mu>1$ such that 
\[
L_{s,H_m}v_s(x)\geq L_{s,F}v_s(x) + \mu  \left(\frac{1}{k_0^q+x}\right)^{2s} v_s\left(\frac{1}{k_0^q+x}\right)
\]
for all $m$ large. 
Using Lemma \ref{scaling} we conclude that there exists a $\lambda>1$ such that $L_{s,H_m}v_s(x)\geq \lambda v_s(x)$, hence $r(L_{s,H_m})\geq \lambda >1$ for all $m$ sufficiently large by Lemma \ref{basic}. It now follows from Theorem \ref{thm:FN} that $\dimh(J_{H_m})>s$ for all $m$ sufficiently large, so $\dimh(J_{F\cup\{k^q\colon k>k_0\}})>s$, which completes the proof.
\end{proof}

As a consequence we find that the final assertion in Theorem \ref{Mq} holds for $q\geq 11$. 
\begin{corollary}
For each $q\geq 11$ we have that $\mathrm{DS}(M_q)$ is the disjoint union of finitely many nontrivial closed intervals. 
\end{corollary}
\begin{proof}
Simply combine Theorems \ref{Mqcritical} and \ref{intervals}.
\end{proof}
To complete the proof of Theorem \ref{Mq} we need to establish the first four assertions concerning  $\mathrm{DS}(M_q)$ where $1\leq q\leq 10$.  We will use the following crude upper bounds for $\dimh(J_{M_q})$, which are easy to obtain,  but  sufficient for our purposes. 
\begin{lemma}\label{ubounds}
We have that  
\begin{align*}
\dimh(J_{M_2}) &\leq 0.67,  & \dimh(J_{M_3})&\leq 0.485,&\dimh(J_{M_4})&\leq 0.38,\\
 \dimh(J_{M_5})&\leq 0.31, & \dimh(J_{M_6})&\leq 0.265, & \dimh(J_{M_7})&\leq 0.234,\\
 \dimh(J_{M_8})&\leq 0.21, & \dimh(J_{M_9})&\leq 0.191, & \dimh(J_{M_{10}})&\leq 0.175.
 \end{align*} 
\end{lemma}  
\begin{proof}
For $m\geq 1$ let $M^m_q=\{1^q,2^q,\ldots, m^q\}$. Let $v_s(x) = (\lambda +x)^{-2s}$ be the eigenvector of $L_{s,\{1\}}$ given  in Lemma \ref{eigenvector} with eigenvalue $\lambda^{-2s}$, where $\lambda = (1+\sqrt{5})/2$. Then 
\begin{eqnarray*} 
L_{s,M_q^m}v_s(x) & = & \lambda^{-2s}v_s(x) +\sum_{n\geq 2}^m \left(\frac{1}{n^q+x}\right)^{2s}v_s\left(\frac{1}{n^q+x}\right)\\
 & = & \lambda^{-2s}\left(\frac{1}{\lambda + x}\right)^{2s} +\sum_{n\geq 2}^m \left(\frac{1}{n^q+x}\right)^{2s}\left(\frac{1}{\lambda + (n^q+x)^{-1}}\right)^{2s}\\
 & \leq & \lambda^{-2s}\left( 1+\sum_{n\geq 2}^m \left(\frac{\lambda +x}{n^q+x}\right)^{2s}\right)v_s(x)\\
& \leq & \lambda^{-2s}\left( 1+\sum_{n\geq 2}^m \left(\frac{\lambda +1}{n^q+1}\right)^{2s}\right)v_s(x)\\
& \leq & \lambda^{-2s}\left( 1+ \left(\frac{\lambda+1}{2^q +1}\right)^{2s} + \left(\frac{\lambda+1}{3^q +1}\right)^{2s} + \left(\frac{\lambda+1}{4^q +1}\right)^{2s} + (\lambda+1)^{2s}\int_4^\infty x^{-2qs}\mathrm{d}x\right)v_s(x)\\
& = & \lambda^{-2s}\left( 1+ \left(\frac{\lambda+1}{2^q +1}\right)^{2s} + \left(\frac{\lambda+1}{3^q +1}\right)^{2s}+ \left(\frac{\lambda+1}{4^q +1}\right)^{2s}  + \frac{(\lambda+1)^{2s}}{2qs-1}4^{-2qs+1}\right)v_s(x).  
\end{eqnarray*}
Now set 
\[
\alpha(q,s) =  \lambda^{-2s}\left( 1+ \left(\frac{\lambda+1}{2^q +1}\right)^{2s} + \left(\frac{\lambda+1}{3^q +1}\right)^{2s} + \left(\frac{\lambda+1}{4^q +1}\right)^{2s} + \frac{(\lambda+1)^{2s}}{2qs-1}4^{-2qs+1}\right).
\]
Note that if $\alpha(q,s)<1$, then  $r(L_{s,M_q^m})<1$ for all $m$, hence $\dimh(J_{M^m_q})< s$ for all $m$ by Theorem \ref{thm:FN}. This implies that $\dimh(J_{M_q})\leq s$ by Theorem \ref{approx}. 
Using a calculator we find that 
\begin{align*}
\alpha(2,0.67)&<0.986, & \alpha(3,0.485)&<0.967, &\alpha(4,0.38)&<0.975,\\
\alpha(5,0.31)&<0.995, & \alpha(6,0.265)&<0.99985, & \alpha(7,0.234)&<0.9983,\\
\alpha(8,0.21)&<0.998, & \alpha(9,0.191)&<0.998, & \alpha(10,0.175)&<0.9989.
\end{align*}
\end{proof}
We should mention that the following much sharper bound, $\dimh(J_{M_2}) <0.59825579$, can be found in \cite[Table 1]{CLU2}.

To begin we  show that the dimension spectrum of $M_q$ is full for $q\in \{1, 2, 3, 4,5\}$, which is statement (i) in Theorem \ref{Mq}. 
\begin{theorem} \label{q=234} For $q\in\{1, 2, 3, 4,5\}$ we have that $\mathrm{DS}(M_q) = [0,\dimh(J_{M_q})]$. 
\end{theorem}
\begin{proof} Given $0< s<\dimh(J_{M_q})$, we will use Lemma \ref{break point} to show that $s\in \mathrm{DS}(M_q)$. Let $n_0^q$ be a strict break point for $(F,s)$, so $n_0>1$. We know that the operator $L_{s, F\cup\{n_0^q\}}$ has a positive eigenvector $v_s$ with eigenvalue $\lambda_s=r(L _{s, F\cup\{n_0^q\}}) \geq 1$. For $m>n_0$, let  $T_{m}=\{(n_0+1)^q, \ldots, m^q\}$ and set $H_m=F\cup T_m$. Then 
\begin{eqnarray*} 
L_{s,H_m}v_s(x) & = & L_{s,F}v_s(x) + \sum_{k=n_0+1}^m \left( \frac{1}{k^q+x}\right)^{2s}v_s\left( \frac{1}{k^q+x}\right)\\
      & \geq & L_{s,F}v_s(x) + \left( \frac{1}{n_0^q+x}\right)^{2s}v_s\left( \frac{1}{n_0^q+x}\right)\sum_{k=n_0+1}^m \left( \frac{n_0^q}{k^q}\right)^{2s},\\
\end{eqnarray*}
as $v_s$ is decreasing by  Theorem \ref{thm:FN}(iv). Set 
\begin{equation}\label{eq:gamma}
\gamma_m=\sum_{k=n_0+1}^m \left( \frac{n_0^q}{k^q}\right)^{2s}.
\end{equation}
If $0<s\leq (2q)^{-1}$, the sum diverges as $m\to\infty$, hence there exists an $M> n_0$ such that $\gamma_M>1$. This implies that there exists a $\mu>1$ such that 
\[
L_{s,H_M}v_s(x)\geq L_{s,F}v_s(x) + \gamma_M\left( \frac{1}{n_0^q+x}\right)^{2s}v_s\left( \frac{1}{n_0^q+x}\right) \geq \mu L_{s,F\cup\{n_0^q\}}v_s(x) \geq \mu v_s(x)
\]  
by Lemma \ref{scaling}. Thus, $r(L_{s,H_M})> 1$, which implies that $\dimh(J_{F\cup \{k^q\colon k>n_0\}}) \geq \dimh(J_{F\cup H_M})> s$, so $s\in \mathrm{DS}(M_q)$ by Lemma \ref{break point}. 

Now if $(2q)^{-1}<s<\dimh(J_{M_q})$, then we consider the following estimate: 
\begin{equation}\label{beta}
\begin{split}
\sum_{k=n_0+1}^\infty \left( \frac{n_0^q}{k^q}\right)^{2s} & \geq   \left(\frac{n_0}{n_0+1}\right)^{2qs} +  \left(\frac{n_0}{n_0+2}\right)^{2qs}+ \left(\frac{n_0}{n_0+3}\right)^{2qs}+n_0^{2qs}\int_{n_0+4}^\infty x^{-2qs}\mathrm{d}x \\
 & = \left(\frac{n_0}{n_0+1}\right)^{2qs} +  \left(\frac{n_0}{n_0+2}\right)^{2qs}+  \left(\frac{n_0}{n_0+3}\right)^{2qs}+ \left(\frac{n_0}{n_0+4}\right)^{2qs} \frac{n_0+4}{2qs-1}=:\gamma(q,n_0,s).
\end{split}
\end{equation} 
Reasoning as above, it suffices to prove that $\gamma(q,n_0,s)>1$. Note that $\gamma(q,n_0,s)$ is decreasing in $s$, and increasing in $n_0$. 

We first consider the case $q=1$. For $n_0\geq  2$ we have that $\gamma(1,n_0,s)\geq \gamma(1,2,1) = 1369/900>1$.
Now consider the case $q=2$. By Lemma \ref{ubounds} we know that $s<\dimh(J_{M_2})\leq 0.67$, and for each $n_0\geq 3$ we have that 
\[
\gamma(2,n_0,s)\geq \gamma(2,3,0.67) \geq 1.3.
\] 
If $n_0=2$,  the estimate $s\leq \dimh(J_{\{1,2^2\}}) \leq 0.4112$ in Table \ref{table} gives  $\gamma(2,2,0.4112)\geq 2.5$. 

The next  case is $q=3$. By Lemma \ref{ubounds} we know that $s<\dimh(J_{M_3})\leq 0.485$, and for each $n_0\geq 3$ we have that 
\[
\gamma(3,n_0,s)\geq \gamma(3,3,0.485) \geq 1.1.
\] 
In case $n_0=2$, the estimate $s\leq \dimh(J_{\{1,2^3\}}) \leq 0.334$ in Table \ref{table} gives  $\gamma(3,2,0.334)\geq 1.5$. 

Now consider the case $q=4$. By Lemma \ref{ubounds} we know that $s<\dimh(J_{M_4})\leq 0.38$, and for each $n_0\geq 3$ we have that 
\[
\gamma(4,n_0,s)\geq \gamma(4,3,0.38) \geq 1.01.
\] 
For  $n_0=2$,  the estimate $s\leq \dimh(J_{\{1,2^4\}}) \leq 0.281$ in Table \ref{table} gives $\gamma(4,2,0.281)\geq 1.14$. 

Finally we need to check the case $q=5$. By Lemma \ref{ubounds} we know that $s<\dimh(J_{M_5})\leq 0.31$, and for each $n_0\geq 4$ we have that 
\[
\gamma(5,n_0,s)\geq \gamma(5,4,0.31) \geq 1.4.
\] 
For $n_0=3$, we have that  $s\leq \dimh(J_{\{1,2^5, 3^5\}}) \leq 0.273$ from Table \ref{table}, which gives $\gamma(5,3,0.273)\geq 1.25$. 
If $n_0=2$, then we cannot use  $\gamma(5,n_0,s)$ and need a different argument. If $n_0=2$, then $F=\{1\}$, hence it suffices to know that 
$\dimh(J_{\{1,2^5\}})<\dimh(J_{M_5\setminus\{2^5\}})$. From the estimates in Table \ref{table} we see that $\dimh(J_{\{1,2^5\}})<\dimh(J_{\{1,3^5,4^5,\ldots,100^5\}})\leq \dimh(J_{M_5\setminus\{2^5\}})$, which completes the proof.

\end{proof}

Next we prove the second statement  in  Theorem \ref{Mq}.
\begin{theorem} \label{q>=6} For $q\geq 6$ we have that 
\begin{equation}\label{ineq}
\dimh(J_{M_q\setminus\{2^q\}})<\dimh(J_{\{1,2^q\}})
\end{equation} 
and  $\mathrm{DS}(M_q)\cap (\dimh(J_{M_q\setminus\{2^q\}}), \dimh(J_{\{1,2^q\}}))$ is empty.
\end{theorem} 
\begin{proof}
For $q\geq 6$ set $s_q=\dimh(J_{\{1,2^q\}})$. We will first consider the case where $q\geq 12$. Recall that $s_q\geq 1.525/q>(2q)^{-1}$ by (\ref{est2^12}) for $q\geq 12$. Let $v_q$ be a positive eigenvector of $L_{s_q,\{1,2^q\}}$ with eigenvalue $1$. Set $H=M_q\setminus\{2^q\}$ and  note that $L_{s,H}$ is a bounded linear operator for all $s>(2q)^{-1}$.
Using  (\ref{ineqv_s}),
\[
\begin{split}
L_{s_q,H}v_q(x) & =  \left(\frac{1}{1+x}\right)^{2s_q}v_q\left(\frac{1}{1+x}\right) + \sum_{n=3}^\infty \left(\frac{1}{n^q+x}\right)^{2s_q}v_q\left(\frac{1}{n^q+x}\right)\\
  &\leq  \left(\frac{1}{1+x}\right)^{2s_q}v_q\left(\frac{1}{1+x}\right) + \left(\frac{1}{2^q+x}\right)^{2s_q}v_q\left(\frac{1}{2^q+x}\right)
  \sum_{n=3}^\infty \left(\frac{2^q+x}{n^q+x}\right)^{2s_q}e^{2s_q\left(\frac{1}{2^q+x}-\frac{1}{n^q+x}\right)}.
 \end{split}
\]

We will now show that $ \sum_{n=3}^\infty \left(\frac{2^q+x}{n^q+x}\right)^{2s_q}e^{2s_q\left(\frac{1}{2^q+x}-\frac{1}{n^q+x}\right)}<1$. 
Note that 
\begin{eqnarray*}
\sum_{n=3}^\infty \left(\frac{2^q+x}{n^q+x}\right)^{2s_q}e^{2s_q\left(\frac{1}{2^q+x}-\frac{1}{n^q+x}\right)}  & \leq & e^{\frac{2s_q}{2^q}} (2^q +1)^{2s_q}
\sum_{n=3}^\infty n^{-2s_qq}\\ 
&  \leq &   e^{\frac{2s_q}{2^q}} (2^q +1)^{2s_q}\int_2^\infty  x^{-2s_qq}\mathrm{d}x\\
  & =  & e^{\frac{2s_q}{2^q}} \left(1 +\frac{1}{2^q}\right)^{2s_q}\left(\frac{2}{2s_qq-1}\right)\\
  &\leq & \frac{2e^{\frac{4s_q}{2^q}}}{2s_qq-1},
\end{eqnarray*}
as $(1+1/n)^n\leq e$ for all $n$. 

The map $s\in ((2q)^{-1},1]\mapsto \frac{2e^{\frac{4s}{2^q}}}{2sq-1}$ is decreasing for all $q\geq 6$, as 
\[
\frac{\mathrm{d}}{\mathrm{d}s}\left(\frac{2e^{\frac{4s}{2^q}}}{2sq-1}\right)= \frac{4e^{\frac{4s}{2^q}}}{(2sq-1)^2}\left( (2sq-1)/2^{q-1} -q\right)\leq  \frac{4e^{\frac{4s}{2^q}}}{(2sq-1)^2}(q/2^{q-2} -q)<0.
\]
Moreover, the map $q\mapsto \frac{2e^{\frac{4s}{2^q}}}{2sq-1}$ is decreasing as well. 

Now using  (\ref{est2^12}), we find that  for $q\geq 12$ that 
\[
\frac{2e^{\frac{4s_q}{2^q}}}{2s_qq-1}\leq \frac{2e^{\frac{4(1.525)}{12\cdot 2^{12}}}}{2(1.525)-1}\leq 0.98<1. 
\]

This implies that 
\[
L_{s_q,H}v_q(x) \leq   \left(\frac{1}{1+x}\right)^{2s_q}v_q\left(\frac{1}{1+x}\right)  + 0.98\left(\frac{1}{2^q+x}\right)^{2s_q}v_q\left(\frac{1}{2^q+x}\right).
\]
By Lemma \ref{scaling} there exists a $\mu<1$ such that $L_{s_q,H}v_q\leq \mu L_{s_q,\{1,2^q\}}v_q =\mu v_q$, hence $r(L_{s_q,H})\leq \mu<1$. It now follows from Lemma \ref{infradius} and Theorem \ref{infdim} that $\dimh(F_H)<s_q=\dimh(J_{\{1,2^q\}})$, as $s_q>(2q)^{-1}$, which completes the proof for  $q\geq 12$.

To deal with the other cases we  use the bounds for $s_q=\dimh(J_{\{1,2^q\}})$ given in Table \ref{table} and the following refined estimate, 
\begin{eqnarray*}
\sum_{n=3}^\infty \left(\frac{2^q+x}{n^q+x}\right)^{2s_q}e^{2s_q\left(\frac{1}{2^q+x}-\frac{1}{n^q+x}\right)}  & \leq & 
	e^{\frac{2s_q}{2^q}} \left( \left(\frac{2^q+1}{3^q+1}\right)^{2s_q} +  \left(\frac{2^q+1}{4^q+1}\right)^{2s_q} +  (2^q+1)^{2s_q}\sum_{n=5}^\infty n^{-2s_qq}\right)\\
 & \leq & 
	e^{\frac{2s_q}{2^q}} \left( \left(\frac{2^q+1}{3^q+1}\right)^{2s_q} +  \left(\frac{2^q+1}{4^q+1}\right)^{2s_q} +  (2^q+1)^{2s_q}\int_4^\infty x^{-2s_qq}\mathrm{d}x\right)\\
	& = &e^{\frac{2s_q}{2^q}} \left( \left(\frac{2^q+1}{3^q+1}\right)^{2s_q} +  \left(\frac{2^q+1}{4^q+1}\right)^{2s_q} +  \left(\frac{2^q+1}{4^q}\right)^{2s_q}\left(\frac{4}{2s_qq-1}\right)\right).
\end{eqnarray*}
Set 
\[
\gamma(s,q) = e^{\frac{2s}{2^q}} \left( \left(\frac{2^q+1}{3^q+1}\right)^{2s} +  \left(\frac{2^q+1}{4^q+1}\right)^{2s} +  \left(\frac{2^q+1}{4^q}\right)^{2s}\left(\frac{4}{2sq-1}\right)\right).
\]
To complete the proof of  inequality (\ref{ineq}), we check for $q\in\{6,\ldots,11\}$   that $\gamma(s_q,q)<1$. 
Using the  upper and lower bounds in Table \ref{table} for $s_q=\dimh(J_{\{1,2^q\}})$ we see that $\gamma(s_{11},11)<0.63$, $\gamma(s_{10},10)<0.67$, $\gamma(s_{9},9)<0.72$, $\gamma(s_{8},8)<0.78$, $\gamma(s_{7},7)<0.85$, and  $\gamma(s_{6},6)<0.96$.

 To show for $q\geq 6$ that $\mathrm{DS}(M_q)\cap (\dimh(J_{M_q\setminus\{2^q\}}), \dimh(J_{\{1,2^q\}}))$ is empty, let $\dimh(J_{M_q\setminus\{2^q\}})<s< \dimh(J_{\{1,2^q\}})$. Suppose by way of contradiction that $\dimh(J_F) =s$ for some $F\subset M_q$. Note that if $2^q\not\in F$, then $F\subset M_q\setminus\{2^q\}$, hence 
 $s\leq \dimh(J_{M_q\setminus\{2^q\}})$, which is impossible. Thus, $2^q\in F$. Now if $1\not\in F$, then $G=(F\setminus\{2^q\})\cup \{1\}\subset M_q\setminus\{2^q\}$. So, Proposition \ref{prop:increasing} gives $s\leq \dimh(J_G)\leq  \dimh(J_{M_q\setminus\{2^q\}})$, which is impossible. So, $\{1,2^q\}\subseteq F$, hence  $\dimh(J_{\{1,2^q\}})\leq s$, which is a contradiction.
\end{proof}

Let us now prove the third statement in Theorem \ref{Mq}. 
\begin{theorem}
For $q\in \{6,7,8\}$ we have that 
\[
 \mathrm{DS}(M_q) = [0,\dimh(J_{M_q\setminus\{2^q\}})]\cup [\dimh(J_{\{1,2^q\}}),\dimh(J_{M_q})].
\] 
\end{theorem}
\begin{proof}
 We will use Lemma \ref{break point}. Suppose  first that $s\in [\dimh(J_{\{1,2^q\}}),\dimh(J_{M_q})]$  and  $n_0^q$ is a strict break point for $(F,s)$, so $n_0\geq 3$.  Reasoning as in the proof of Theorem \ref{q=234} we see that it suffices to show for $(2q)^{-1}<s<\dimh(J_{M_q})$ that $\gamma(q,n_0,s)>1$ in (\ref{beta}). 
 If $n_0\geq 4$, then  using the estimates in Lemma \ref{ubounds} we find that 
 \[
\gamma(6,4,0.265)> 1.3,\quad \gamma(7,4.0.234)> 1.2, \mbox{ and }\gamma(8,4,0.21)>1.2.
\]
On the other hand if $n_0=3$, then we know that $s\leq \dimh(J_{\{1,2^q,3^q\}})$ and we can use the upper bounds in Table \ref{table} to get  that 
 \[
\gamma(6,3,0.238626)> 1.13, \mbox{ and } \gamma(7,3,0.212933)> 1.05.
\]
For $q=8$, we need to expand the sum on the left-hand-side in (\ref{beta}) and consider 
\begin{multline*}
\gamma'(q,n_0,s):=\left(\frac{n_0}{n_0+1}\right)^{2qs} +  \left(\frac{n_0}{n_0+2}\right)^{2qs}+  \left(\frac{n_0}{n_0+3}\right)^{2qs}
+  \left(\frac{n_0}{n_0+4}\right)^{2qs}\\+  \left(\frac{n_0}{n_0+5}\right)^{2qs}+ \left(\frac{n_0}{n_0+6}\right)^{2qs} \frac{n_0+6}{2qs-1},
\end{multline*}
which satisfies $\gamma'(8,3,0.192786)>1.004$. 

If $s\in [0,\dimh(J_{M_q\setminus\{2^q\}})]$, then we can use Lemma \ref{break point} with respect to $A=M_q\setminus\{2^q\}$. So, if $n_0^q$ is a strict break point for $(F,s)$, then $n_0\geq 3$. Now the same inequalities for $\gamma(q,n_0,s)$ and $\gamma'(q,n_0,s)$ as above imply that $s\in \mathrm{DS}(M_q\setminus\{2^q\})\subset \mathrm{DS}(M_q)$. 
\end{proof}

To complete the proof of Theorem \ref{Mq} it remains to show the fourth assertion. 
\begin{theorem}
For $q\in\{9,10\}$ we have that  $\dimh(J_{M_q\setminus\{3^q\}}) <\dimh(J_{\{1,2^q,3^q\}})$ and 
\[
 \mathrm{DS}(M_q) = [0,\dimh(J_{M_q\setminus\{2^q\}})]\cup [\dimh(J_{\{1,2^q\}}),\dimh(J_{M_q\setminus\{3^q\}})]\cup [\dimh(J_{\{1,2^q,3^q\}}),\dimh(J_{M_q})].
\] 
\end{theorem}
\begin{proof} To establish the inequality we reason as in the proof of Theorem \ref{q>=6}. Let $s_q=\dimh(J_{\{1,2^q,3^q\}})$ and  $v_q$ be the strictly positive eigenvector of $L_{s_q,\{1,2^q,3^q\}}$ with eigenvalue 1. So, $s_q> 1.525/q>(2q)^{-1}$ by (\ref{est2^12}) and $L_{s,H}$, with $H=M_q\setminus\{3^q\}$, is a bounded linear operator for $s>(2q)^{-1}$. 
Using (\ref{ineqv_s}), 
\[
\begin{split}
L_{s_q,H}v_q(x) & =  \left(\frac{1}{1+x}\right)^{2s_q}v_q\left(\frac{1}{1+x}\right) +  \left(\frac{1}{2^q+x}\right)^{2s_q}v_q\left(\frac{1}{2^q+x}\right) + 
\sum_{n=4}^\infty \left(\frac{1}{n^q+x}\right)^{2s_q}v_q\left(\frac{1}{n^q+x}\right)\\
  &\leq  \left(\frac{1}{1+x}\right)^{2s_q}v_q\left(\frac{1}{1+x}\right) +  \left(\frac{1}{2^q+x}\right)^{2s_q}v_q\left(\frac{1}{2^q+x}\right)\\
   & \qquad \qquad+ \left(\frac{1}{3^q+x}\right)^{2s_q}v_q\left(\frac{1}{3^q+x}\right)
  \sum_{n=4}^\infty \left(\frac{3^q+x}{n^q+x}\right)^{2s_q}e^{2s_q\left(\frac{1}{3^q+x}-\frac{1}{n^q+x}\right)}.
 \end{split}
\]
Note that  for $k\geq 4$ we have that 
\begin{multline*}
\sum_{n=4}^\infty \left(\frac{3^q+x}{n^q+x}\right)^{2s_q}e^{2s_q\left(\frac{1}{3^q+x}-\frac{1}{n^q+x}\right)}  \\ \leq  
	e^{\frac{2s_q}{3^q}} \left( \left(\frac{3^q+1}{4^q+1}\right)^{2s_q} +\cdots+  \left(\frac{3^q+1}{k^q+1}\right)^{2s_q} +  (3^q+1)^{2s_q}\sum_{n=k+1}^\infty n^{-2s_qq}\right)\\
  \leq  
	e^{\frac{2s_q}{3^q}} \left( \left(\frac{3^q+1}{4^q+1}\right)^{2s_q} +  \cdots +\left(\frac{3^q+1}{k^q+1}\right)^{2s_q} +  (3^q+1)^{2s_q}\int_k^\infty x^{-2s_qq}\mathrm{d}x\right)\\
 = e^{\frac{2s_q}{3^q}} \left( \left(\frac{3^q+1}{4^q+1}\right)^{2s_q} +\cdots +  \left(\frac{3^q+1}{k^q+1}\right)^{2s_q} +  \left(\frac{3^q+1}{k^q}\right)^{2s_q}\left(\frac{k}{2s_qq-1}\right)\right)=: \beta(s_q,q,k).
\end{multline*}
Using the upper and lower bounds for $s_q$ in Table \ref{table} and taking $k=8$, we find that $\beta(s_9,9,8)< 0.99$ and $\beta(s_{10},10,8)<0.94$. 
It now follows from Lemma \ref{scaling}  that there exists a $\mu<1$ such that $L_{s_q,H}v_q\leq \mu L_{s_q,\{1,2^q,3^q\}}v_q =\mu v_q$. So,  $r(L_{s_q,H})\leq \mu<1$, which implies that $\dimh(F_H)<s_q=\dimh(J_{\{1,2^q,3^q\}})$ by Lemma \ref{infradius} and Theorem \ref{infdim}, as $s_q>(2q)^{-1}$. 

Reasoning in the same way as in the proof of Theorem \ref{q>=6} it can easily be shown for $q=9$ and $q=10$ that there is no $s\in\mathrm{DS}(M_q)$ between the closed intervals. 

 To show that each element in the  intervals belongs to the dimension spectrum we will use Lemma \ref{break point}. Suppose  first that $s\in [\dimh(J_{\{1,2^q,3^q\}}),\dimh(J_{M_q})]$  and  $n_0^q$ is a strict break point for $(F,s)$, so $n_0\geq 4$.  Using the same arguments as in the proof of Theorem \ref{q=234} we see that it suffices to show for $(2q)^{-1}<s<\dimh(J_{M_q})$ that $\gamma(q,n_0,s)>1$ in (\ref{beta}) to conclude that $s\in \mathrm{DS}(M_q)$. 
 If $n_0\geq 4$, we can use the upper bounds in Lemma \ref{ubounds} to get that $\gamma(9,4,0.191)> 1.1$ and $\gamma(10,4,0.175)>1.1$.

On the other hand, if $s\in [\dimh(J_{\{1,2^q\}}),\dimh(J_{M_q\setminus\{3^q\}})]$, we can apply  Lemma \ref{break point} with $A = M_q\setminus\{3^q\}$. In that case, if $n_0^q$ is a strict break point for $(F,s)$, then $n_0\geq 4$, and the same estimates as above hold. So, $s\in \mathrm{DS}(M_q\setminus\{3^q\})\subset \mathrm{DS}(M_q)$. 
Finally, for $s\in [0,\dimh(J_{M_q\setminus\{2^q\}})]$ we apply  Lemma \ref{break point} with $A = M_q\setminus\{2^q\}$. So, if  $n_0^q$ is a strict break point for $(F,s)$, then $n_0\geq 3$. Using the upper bound for $\dimh(J_{\{1,2^q\}})$ in Table \ref{table}  for $q=9$ and $q=10$, we  get  that 
 \[
\gamma(9,3,0.162510)> 1.09, \mbox{ and } \gamma(10,3,0.150820)> 1.02.
\]
It follows that  $s\in \mathrm{DS}(M_q\setminus\{2^q\})\subset \mathrm{DS}(M_q)$ and we are done. 
\end{proof}

\section{Final remarks and open problems} 
We conclude the paper by briefly discussing some open questions that follow from this work.  To begin, from Theorem \ref{thm2} we know that for $q\geq 3$  the dimension spectrum contains  $[0,\ln (2)/(2\ln q)]$ and is nowhere dense in $[\mu^1,  \dim_{\mathcal{H}}(J_{P_q^*})]$. At present the exact structure of the dimension spectrum  in the interval $(\ln (2)/(2\ln q),\mu^1)$ is unclear for $q\geq 3$, but we believe that the dimension spectrum is nowhere dense there.  Understanding the structure of the dimension spectrum in this interval would give a complete description of the dimension spectrum of $P^*_q$ for $q\geq 3$.

There are a number of natural questions regarding Theorem \ref{Mq}. Firstly, it seems that the number of intervals increases with $q$. We have, however, not been  able to establish  if the number of intervals keeps growing when $q$ gets larger, or, if there exists an a priori upper bound for the number of distinct intervals.   If the number of intervals keeps growing, it  would  be interesting to understand at which values of $q$ the number of intervals in the dimension spectrum of $M_q$ jumps. For instance, the first jump from one to two intervals occurs at $q=6$, and at $q=9$ it jumps from two to three intervals.  If we consider $q$ as a continuous parameter, one could  ask at which value of $q\in (5,6)$ the number of intervals in the dimension spectrum of $M_q$ changes from one to two. Our methods rely on estimates of the Hausdorff dimension of particular sets and  do not allow us to address these questions at present. 

It is worth noting that in \cite[Section 5]{DS} a number of conjectures concerning the dimension spectrum are formulated. One could add the following general questions to them. If $A$ is an infinite subset of $\mathbb{N}$ such that $\mathrm{DS}(A)$ contains two solid closed intervals $[a,b]$ and $[c,d]$, with $a<b<c<d$, and $\mathrm{DS}(A)$ is nowhere dense in $(b,c)$, is true that $\mathrm{DS}(A)\cap (b,c)$ is empty? It also seems reasonable to speculate that  if $A\subseteq \mathbb{N}$ with finiteness parameter $\sigma_0 =0$ and $\mathrm{DS}(A)$ contains a solid closed interval $[a,b]$ with $a>0$, then there exists a $\delta>0$ such that $[0,\delta]\subset \mathrm{DS}(A)$.

\section{Appendix}
The statement of Lemma \ref{Fcupb} holds in greater generality, but for simplicity we present it here in the setting of continued fraction expansions. 
\begin{proof}[Proof of Lemma \ref{Fcupb}]
Note that to establish (\ref{Acupn}) it suffices to show that there exists a constant $C_F>1$ such that (\ref{Acupn}) holds for all $n$ sufficiently large. 
Let $v_s$ be the strictly positive eigenvector of $L_{s,F}$ with eigenvalue $\lambda_s=r(L_{s,F})$, and let $w_s$ be the strictly positive eigenvector of $L_{s,F\cup\{n\}}$ with eigenvalue $\mu_s=r(L_{s,F\cup\{n\}})$ for $\sigma\leq s<1$. If we can show that there exists a constant $C_1>1$ such that for all $n$ sufficiently large, $\mu_s<1$ for $s = \sigma + C_1n^{-2\sigma}$, then we know by Theorem \ref{thm:FN} that $\dimh(J_{F\cup\{n\}})<  \sigma + C_1n^{-2\sigma}$ for all $n$ large. 

By (\ref{ineqv_s}) we know that 
\[
v_s\left(\frac{1}{n+x}\right) \leq v_s(x)e^{2s}
\] 
for all $x\in [0,1]$. Thus, 
\[
(L_{s,F\cup\{n\}}v_s)(x)\leq \lambda_s v_s(x) +n^{-2s}v_s(x)e^{2s} = (\lambda_s +n^{-2s}e^{2s})v_s(x), 
\]
so that $r(L_{s,F\cup\{n\}}) \leq \lambda_s+ n^{-2s}e^{2s}$.  

For $n\in \mathbb{N}$ let $\theta_n\colon x\mapsto \frac{1}{n+x}$. We know  for $s>\sigma$, that 
$\left((\theta_n\circ\theta_m)'(x)\right)^{s-\sigma} \leq 4^{-(s-\sigma)}$, see (\ref{1/4}) for all $x\in [0,1]$. Thus, 
\begin{eqnarray*}
(L^2_{s, F}v_\sigma)(x) &= & \sum_{n,m\in A} ((\theta_n\circ\theta_m)'(x))^sv_\sigma((\theta_n\circ\theta_m)(x)) \\
 & \leq & 4^{-(s-\sigma)}\sum_{n,m\in A} ((\theta_n\circ\theta_m)'(x))^\sigma v_\sigma((\theta_n\circ\theta_m)(x)) = 4^{-(s-\sigma)}v_\sigma(x),
\end{eqnarray*}
which gives $ r(L^2_{s, F}) \leq 4^{-(s-\sigma)}$, hence $\lambda_s=r(L_{s, F})\leq 2^{-(s-\sigma)}$. 
Thus,   $r(L_{s,F\cup\{n\}}) \leq 2^{-(s-\sigma)} + n^{-2s}e^{2s}$ and we see that $\mu_s<1$ if $2^{-(s-\sigma)} + n^{-2s}e^{2s}<1$. As $2^{s-\sigma}e^{2s}<e^3$,  this inequality holds if 
\begin{equation}\label{ineq1}
n^{-2\sigma}e^{3}<2^{s-\sigma} -1. 
\end{equation}

We now wish to show that there exists a $C_1>1$ such that  $s=\sigma+C_1n^{-2\sigma}$ satisfies (\ref{ineq1}) and $\sigma<s<1$. Note that (\ref{ineq1}) holds if 
\begin{equation}\label{ineq2}
n^{-2\sigma}e^3< 2^{C_1n^{-2\sigma}}-1 = e^{C_1n^{-2\sigma}\ln(2)} -1. 
\end{equation}
As $e^x-1>x$ for $x>0$, we see that (\ref{ineq2}) holds if $n^{-2\sigma}e^3< C_1n^{-2\sigma}\ln(2)$,
which gives $C_1> \frac{e^3}{\ln(2)}$. To ensure that $s<1$ for  $s= \sigma + C_1n^{-2\sigma}$, we also require that $n>\left(\frac{C_1}{1-\sigma}\right)^{1/2\sigma}$.
Thus, for all $n$ sufficiently large, $\mu_s<1 $ for $s= \sigma + C_1n^{-2\sigma}$, where $C_1> \frac{e^3}{\ln(2)}$, which establishes the upper bound for 
$\dimh(J_{F\cup\{n\}})$. 

To show that $\lim_n \dimh(J_{F\cup\{n\}}) =0$ for $|F| =1$, we note that if $|F|=1$, then $\sigma =0$. So, $\mu_s<1$ if $2^{-s} + n^{-2s}e^{2s}<1$ in that case, which is equivalent to $e/n<(1-2^{-s})^{1/2s}$. Clearly for each $\epsilon>0$, there exists an $N>1$ such that $e/n<(1-2^{-\epsilon})^{1/2\epsilon}$ for all $n>N$, hence $\mu_\epsilon<1$ for all $n>N$.  Now Theorem \ref{thm:FN} implies that  
$\dimh(J_{F\cup\{n\}})\to 0$ as $n\to\infty$.

To obtain the lower bound for $\dimh(J_{F\cup\{n\}})$, we need the  fact that $s\mapsto \ln\mu_s$ is strictly decreasing and convex, see for instance \cite[Theorem 8.1]{FN2}.
If we can show that there exists a constant $C_2<1$ such that for all $n$ sufficiently large, $\mu_s>1$ for $s = \sigma + C_2n^{-2\sigma}$, then it follows  from Theorem \ref{thm:FN} that $\dimh(J_{F\cup\{n\}})>  \sigma + C_2n^{-2\sigma}$ for all $n$ large. 

Using the Mean Value Theorem we know for $0\leq y\leq z\leq 1$ that 
\[
\ln\left(\frac{n+z}{n+y}\right)^2= 2(\ln(n+z)-\ln(n+y))\leq \frac{2}{n}(z-y),
\]
so 
\[
\left(\frac{1}{n+y}\right)^2\leq \left(\frac{1}{n+z}\right)^2e^{\frac{2}{n}(z-y)}.
\]
It follows that $n^{-2}e^{-2}\leq (n+x)^{-2}$ for $x\in [0,1]$.
We also know from (\ref{ineqv_s})  that 
\[
e^{-2}v_\sigma(x)\leq v_\sigma\left(\frac{1}{n+x}\right)\mbox{\quad for } x\in [0,1].
\]
Thus, 
\[
n^{-2\sigma}e^{-4} v_\sigma(x)\leq \left(\frac{1}{n+x}\right)^{2\sigma}v_\sigma\left(\frac{1}{n+x}\right), 
\]
so that 
\[
L_{\sigma, F\cup\{n\}}v_\sigma(x) = v_\sigma(x) + \left(\frac{1}{n+x}\right)^{2\sigma}v_\sigma\left(\frac{1}{n+x}\right) \geq (1+n^{-2\sigma}e^{-4})v_\sigma(x),
\]
hence $\mu_\sigma\geq 1+n^{-2\sigma}e^{-4}$. 

Let $u$ be the constant $1$ function on $[0,1]$. Then $L_{0,F\cup\{n\}} u = (|F|+1) u$, hence $r(L_{0,F\cup\{n\}}) = |F|+1$. Set $\rho(s) = \ln \mu_s$, which is a strictly decreasing convex function with $\rho(0)= \ln (|F|+1)>\rho(\sigma)\geq \ln ( 1+n^{-2\sigma}e^{-4})>0$. Let $s_1>\sigma$ be the unique value such that $\rho(s_1)=0$. 
The straight-line through $(0, \ln(|F|+1))$ and $(\sigma, \ln(1+n^{-2\sigma}e^{-4}))$ intersects the $s$-axis at say $s_2$ with $\sigma<s_2\leq s_1$ by convexity.  
A simple computation gives
\[
s_2 = \sigma\left(\frac{\ln (|F|+1)}{\ln(|F|+1) - \ln(1+n^{-2\sigma}e^{-4})}\right)> \sigma\left(1+\frac{\ln(1+n^{-2\sigma}e^{-4})}{\ln(|F|+1)}\right).
\]
Using the power series for the function $x\mapsto \ln(1+x)$ for $0\leq x<1$, we find that 
\[
s_2 > \sigma\left(1+\frac{n^{-2\sigma}e^{-4}  -  \frac{1}{2}(n^{-2\sigma}e^{-4})^2}{\ln(|F|+1)}\right)\geq \sigma+\frac{\sigma}{2e^4\ln(|F|+1)}n^{-2\sigma}.
\]
Thus, if we take $C_2 = \frac{\sigma}{2e^4\ln(|F|+1)}<1$ and set $s = \sigma + C_2n^{-2\sigma}$, we have that $\ln(\mu_s)>0$, hence $\mu_s>1$. 

Taking $C_F =\max\{C_1,C_2^{-1}\}>1$, we conclude that $\sigma + C_F^{-1}n^{-2\sigma}<\dimh(J_{F\cup\{n\}})<  \sigma + C_Fn^{-2\sigma}$ for all $n$ large, which completes the proof.

\end{proof}

\end{document}